\documentclass[intlim,righttag,10pt]{amsart}
\usepackage{bbm}
\usepackage{amscd}
\usepackage{amssymb}
\usepackage[all]{xy}
\RequirePackage{mathrsfs}

\def\invlim{\mathop{\vtop{\ialign{##\crcr$\hfill{\lim}\hfil$\crcr
\noalign{\kern1pt\nointerlineskip}\leftarrowfill\crcr\noalign
{\kern -3pt}}}}\limits}
\def\dirlim{\mathop{\vtop{\ialign{##\crcr$\hfill{\lim}\hfil$\crcr
\noalign{\kern1pt\nointerlineskip}\rightarrowfill\crcr\noalign
{\kern -3pt}}}}\limits}
\def\lomapr#1{\smash{\mathop{\relbar\joinrel\longrightarrow}\limits^{#1}}}
 \def\verylomapr#1{\smash{\mathop{\relbar\joinrel\relbar\joinrel\relbar\joinrel\longrightarrow}\limits^{#1}}}
\def\veryverylomapr#1{\smash{\mathop{\relbar\joinrel\relbar\joinrel\relbar
\joinrel\relbar\joinrel\relbar\joinrel\longrightarrow}\limits^{#1}}}
\def\phi{\varphi}
\def\epsilon{\varepsilon}
\let\mathcal\mathscr

\newtheorem{theorem}{Theorem}[section]
 \newtheorem{lemma}[theorem]{Lemma}
 \newtheorem{proposition}[theorem]{Proposition}
 \newtheorem{corollary}[theorem]{Corollary}

\theoremstyle{definition}
\newtheorem{definition}[theorem]{Definition}
\newtheorem{remark}[theorem]{Remark}

\newtheorem{example}[theorem]{Example}
\newtheorem*{acknowledgments}{Acknowledgments}

\numberwithin{equation}{section}

\newcommand\Vtextvisiblespace[1][.3em]{%
  \mbox{\kern.06em\vrule height.3ex}%
  \vbox{\hrule width#1}%
  \hbox{\vrule height.3ex}}

\newcommand{\Qp}{\mathbf{Q}_p}

\renewcommand{\phi}{\varphi}

\newcommand{\R}{\mathrm {R} }

\newcommand{\G}{\mathrm {DF}_{\overline{K}} }

\newcommand{\pst}{pst}

  \newcommand{\ph}{\operatorname{pH}}
\newcommand{\ovk}{\overline{K} }

\newcommand{\ad}{\operatorname{ad} }

\newcommand{\Pic}{\operatorname{Pic} }

 \newcommand{\coker}{\operatorname{coker} }

 \newcommand{\Sym}{\operatorname{Sym} }
  
 \newcommand{\holim}{\operatorname{holim} }
  \newcommand{\Jac}{\operatorname{Jac} }

 \newcommand{\eet}{\operatorname{\acute{e}t} }

 \newcommand{\nr}{\operatorname{nr} }
 
 \newcommand{\Spec}{\operatorname{Spec} }

 \newcommand{\Hom}{\operatorname{Hom} }

 \newcommand{\Ext}{\operatorname{Ext} }

 \newcommand{\Gal}{\operatorname{Gal} }
 \newcommand{\tr}{ \operatorname{tr} }
 \newcommand{\can}{ \operatorname{can} }
 \newcommand{\Mod}{ \operatorname{Mod} }  
 \newcommand{\Bun}{ \operatorname{Bun} }
  \newcommand{\Banach}{ \operatorname{Banach} }
    
 \newcommand{\id}{ \operatorname{Id} }
\newcommand{\synt}{ \operatorname{syn} }
 
 \newcommand{\Cone}{\operatorname{Cone} }
 \newcommand{\Proj}{\operatorname{Proj} }

\newcommand{\st}{\operatorname{st} }

\newcommand{\hk}{\operatorname{HK} }
\newcommand{\dr}{\operatorname{dR} }
\newcommand{\rig}{\operatorname{rig} }
 \newcommand{\ff}{\operatorname{FF} }
 
 \newcommand{\kker}{\operatorname{Ker} }
 \newcommand{\crr}{\operatorname{cr} }
 \newcommand{\gr}{\operatorname{gr} }
 \newcommand{\im}{\operatorname{Im} }

 \newcommand{\kr}{^{\scriptscriptstyle\bullet}}

 \newcommand{\Modif}{\mathcal{M}}
 \newcommand{\Fmod}{\operatorname{FMod}}
   
   \newcommand{\DF}{\operatorname{DF}}

 \newcommand{\sh}{{\mathcal{H}}}

 \newcommand{\srr}{{\mathcal{R}}}

 \newcommand{\scc}{{\mathcal{C}}}

 \newcommand{\so}{{\mathcal O}}
 
 \newcommand{\se}{{\mathcal{E}}}

 \newcommand{\sx}{{\mathcal{X}}}
 \newcommand{\sss}{{\mathcal{S}}}
\newcommand{\sd}{{\mathcal{D}}}
\newcommand{\sm}{{\mathcal{M}}}

 \newcommand{\Z}{ {\mathbf Z} }
    
   \newcommand{\Q}{ {\mathbf Q}}
   \newcommand{\N}{{\mathbf N}}
   \newcommand{\BC}{{\mathcal B}{\mathcal C}}   
   
   \def\B{{\bf B}}
   \newcommand{\un}{\mathbbm 1} 
    
\numberwithin{equation}{section}

\begin{document}
 \title[Geometric syntomic cohomology and vector bundles on the Fargues-Fontaine curve]{Geometric syntomic  cohomology and vector bundles on the Fargues-Fontaine curve}
 \author{Wies{\l}awa Nizio{\l}}
 \date{\today}
\thanks{The author's research was supported in part by the grant ANR-14-CE25.}
 \email{ wieslawa.niziol@ens-lyon.fr}
 \begin{abstract}
 We show that geometric syntomic cohomology lifts canonically to the category of Banach-Colmez spaces and study its relation to extensions of modifications of vector bundles on the Fargues-Fontaine curve. We include some computations of geometric syntomic cohomology Spaces:  they are finite rank $\Qp$-vector spaces for ordinary varieties, but in the nonordinary case, these cohomology Spaces carry much more information, in particular they 
can have a non-trivial $C$-rank. This dichotomy is reminiscent of the Hodge-Tate period map for $p$-divisible groups.
 \end{abstract}
 \maketitle
 \tableofcontents
\section{Introduction}

  As is well-known, syntomic cohomology is a $p$-adic analog of Deligne-Beilinson cohomology. Recall that the latter is an absolute Hodge cohomology \cite{BE0}, i.e., it can be computed as $\Ext$ groups in the category of mixed Hodge structures; a feature that makes definition of
regulators straightforward.
  In \cite{DN}, it was shown that this is also the case for syntomic cohomology of varieties over a $p$-adic local filed $K$:
 it is an absolute $p$-adic Hodge cohomology, i.e., it can be computed as $\Ext$ groups in a category of admissible filtered $\phi$-modules of Fontaine. 
In this paper we prove an analog of this statement for geometric syntomic cohomology (i.e. for varieties over $\overline K$):
 it can be computed as $\Ext$ groups in the category of effective filtered $\phi$-modules over $\overline K$
(what Fargues would call ``$\varphi$-modules jaug\'es''). 
It follows that it has an extra rigid structure, namely, it comes from a complex of finite dimensional Banach-Colmez spaces; hence the geometric syntomic cohomology groups
are finite dimensional Banach-Colmez spaces.
In the case of ordinary varieties, these groups are of Dimension $(0,h)$, i.e. are finite dimensional $\Qp$-vector
spaces, but in the nonordinary case, these groups can have Dimension $(d,h)$, with $d\geq 1$, and thus carry much more information,
as we show on the example of the symmetric square of an elliptic curve.

We are now going to explain in more details what we have said above.
Recall that,
  for a log-smooth variety $\sx$ over $\so_K$ -- a complete discrete valuation ring of mixed characteristic $(0,p)$ with field of fractions $K$ and perfect residue field --  the arithmetic syntomic cohomology of $X$ is defined as a filtered Frobenius eigenspace of its crystalline cohomology
  \begin{equation}
  \label{syntomic-ar}
  \R\Gamma_{\synt}(\sx,r):=[\R\Gamma_{\crr}(\sx)^{\phi=p^r}\to \R\Gamma_{\crr}(\sx)/F^r],\quad r\geq 0.
  \end{equation}
    For a variety $X$ over $K$, this sheafifies well in the $h$-topology \footnote{Contrary to  crystalline cohomology itself which does not sheafify well. The sheafification process uses the fact that $h$-topology has a basis consisting of smooth varieties with semistable compactifications \cite{Be1}.}  and yields syntomic cohomology  $\R\Gamma_{\synt}(X,r)$, $r\geq 0$, of $X$ \cite{NN}. This cohomology comes equipped with a period map to \'etale cohomology
$$
\rho_{\synt}: \R\Gamma_{\synt}(X,r)\to\R\Gamma_{\eet}(X,\Q_p(r))
$$
that is a quasi-isomorphism after taking the truncation $\tau_{\leq r}$. Syntomic cohomology approximates better $p$-adic motivic cohomology than \'etale cohomology does; in particular,  \'etale $p$-adic regulators from $K$-theory  factor through syntomic cohomology. 

 As was shown in \cite{DN}, syntomic cohomology is an absolute p-adic Hodge cohomology. Namely, the data of the Hyodo-Kato cohomology $\R\Gamma_{\hk}(X_{\ovk})$ and  the de Rham cohomology $\R\Gamma_{\dr}(X_{\ovk})$ together with the Hyodo-Kato quasi-isomorphism $\iota_{\dr}:\R\Gamma_{\hk}(X_{\ovk})\otimes_{F^{\nr}}\ovk\stackrel{\sim}{\to}\R\Gamma_{\dr}(X_{\ovk})$, where $F$ is the maximal absolutely unramified subfield of $K$ and $F^{\nr}$ -- its maximal unramified extension,  allows to canonically associate to any variety $X$ over $K$ a complex $\R\Gamma_{\DF_K}(X,r)$ of Fontaine's admissible filtered $(\phi,N,G_K)$-modules. One proves that 
    $$
    \R\Gamma_{\synt}(X,r)\simeq\Hom_{D^b(\DF_K)}(\un,\R\Gamma_{\DF_K}(X,r)),\quad r\geq 0. 
        $$
  
   Syntomic cohomology from (\ref{syntomic-ar}) has a geometric version. Geometric syntomic cohomology is defined as a filtered Frobenius eigenspace of  geometric crystalline cohomology 
  \begin{equation}
  \R\Gamma_{\synt}(\sx_{\so_{\ovk}},r):=[\R\Gamma_{\crr}(\sx_{\so_{\ovk}})^{\phi=p^r}\to \R\Gamma_{\crr}(\sx_{\so_{\ovk}})/F^r],\quad r\geq 0,
  \end{equation}
  where $\ovk$ is an algebraic closure of $K$ and $\so_{\ovk}$ -- its ring of integers.
   For a variety $X$ over $\ovk$, this also sheafifies well in the $h$-topology\footnote{Here  crystalline cohomology itself  also sheafifies well and yields well-behaved crystalline cohomology $\R\Gamma_{\crr}(X)$ of $X$.} and yields syntomic cohomology of $X$ \cite{NN} $$\R\Gamma_{\synt}(X,r)=
   [\R\Gamma_{\crr}(X)^{p^r=\phi}\to \R\Gamma_{\crr}(X)/F^r].$$
   This cohomology comes equipped with a period map to \'etale cohomology
$$
\rho_{\synt}: \R\Gamma_{\synt}(X,r)\to\R\Gamma_{\eet}(X,\Q_p(r))
$$
that is a quasi-isomorphism after taking the truncation $\tau_{\leq r}$. In particular, the groups $H^i_{\synt}(X,r)$, $i\leq r$, are finite rank $\Qp$-vector spaces.
   
   The first main result of this paper is that geometric syntomic is an absolute $p$-adic Hodge cohomology. To explain what this means, we note that 
   we have the  isomorphisms
   $$
   H^i_{\crr}(X)\simeq H^i_{\hk}(X)\otimes_{F^{\nr}}\B^+_{\crr},\quad 
   H^i_{\crr}(X)/F^r      \simeq (H^i_{\dr}(X)\otimes_{\ovk}\B^+_{\dr})/F^r,  $$
 where $\B^*_*$ denotes rings of $p$-adic periods, and  
    an isomorphism
   $$ \iota_{\dr}: H^i_{\crr}(X)\otimes_{\B^+_{\crr}}\B^+_{\dr}\stackrel{\sim}{\to}H^i_{\dr}(X)\otimes_{\ovk}\B^+_{\dr}.      $$
 Hence  the data
     $$
     (H^i_{\crr}(X)\otimes_{\B^+_{\crr}}
     \B^+,(H^i_{\dr}(X)\otimes_{\ovk}\B^+_{\dr},F^r),\iota_{\dr})
     $$ gives  an effective filtered $\phi$-module over $\ovk$\footnote{$\phi$-module jaug\'e in the original terminology of Fargues.}. The category of such objects is equivalent to the category of  effective modifications of vector bundles on the Fargues-Fontaine curve\footnote{However the equivalence does not preserve relevant exact structures.}. For the above data: the vector bundle is associated to the $\phi$-module over $\B^+_{\crr}$ given by $H^i_{\crr}(X)$  and it is modified at infinity by the $\B^+_{\dr}$-lattice $F^r(H^i_{\dr}(X)\otimes_{\ovk}\B^+_{\dr})$.
    As in the arithmetic case this data can be lifted to complexes and we obtain the first main result of this paper.
 \begin{theorem}
 \label{first}
 \begin{enumerate}
 \item   To every variety $X$ over $\ovk$ one can associate a canonical complex of effective filtered $\phi$-modules $\R\Gamma_{\G}(X,r)$,$r\geq 0$.
\item There is a canonical quasi-isomorphism
$$
\R\Gamma_{\synt}(X,r)\simeq \R\Hom_+(\un,\R\Gamma_{\G}(X,r)),\quad r\geq 0,
$$
where $\R\Hom_+$ denotes the derived $\Hom$ in the category of effective filtered $\phi$-modules over $\ovk$.
\end{enumerate}
\end{theorem}
  For an effective $\phi$-module $M$, the complex $\R\Hom_+(\un,M)$ has nontrivial cohomology only in degrees $0,1$; we call them $H^0_+(\ovk,M)$ and $H^1_+(\ovk,M)$. Hence the above spectral sequence reduces  to  the short exact sequence
\begin{equation}
\label{seq-intr}
0\to H^1_+(\ovk,H^{i-1}_{\G}(X,r))\to H^i_{\synt}(X,r)\to H^0_+(\ovk,H^i_{\G}(X,r))\to 0.
\end{equation}
This sequence can be easily seen to arise from {\em the fundamental (long) exact sequence}
\begin{equation}
     \label{exact-intr}
     \to (H^{i-1}_{\dr}(X)\otimes_{\ovk}\B^+_{\dr})/F^r\to H^i_{\synt}(X,r)\to (H^i_{\hk}(X)\otimes_{F^{\nr}}\B^+)^{\phi=p^r}\lomapr{\iota_{\dr}}
     (H^i_{\dr}(X)\otimes_{\ovk}\B^+_{\dr})/F^r\to
     \end{equation}
     The terms 
     $$
      (H^i_{\hk}(X)\otimes_{F^{\nr}}\B^+)^{\phi=p^r},\quad     (H^i_{\dr}(X)\otimes_{\ovk}\B^+_{\dr})/F^r
     $$
     are the key examples of $p$-adic Banach spaces that are $C$-points, $C=\widehat{\ovk}$, of finite dimensional Banach-Colmez spaces \cite{CB}. The latter are defined, roughly, as finite rank $C$-vector spaces modulo finite rank $\Q_p$-vector spaces, rigidified as  functors on a subcategory of perfectoid spaces. The second main result of this paper states that this is also the case for  geometric syntomic cohomology.
  \begin{theorem}
  \label{second}
  There exists a canonical complex of Banach-Colmez spaces ${\mathbb R}\Gamma_{\synt}(X,r)$
  such that
  \begin{enumerate}
  \item ${\mathbb R}\Gamma_{\synt}(X,r)(C)\simeq \R\Gamma_{\synt}(X,r)$; in particular, $  H^i{\mathbb R}\Gamma_{\synt}(X,r)(C)\simeq H^i_{\synt}(X,r)$.
  \item The fundamental exact sequence (\ref{exact-intr})
     lifts canonically to the category of Banach-Colmez spaces.
  \item The syntomic period map \cite{NN}
  $$\rho_{\synt}:\R\Gamma_{\synt}(X,r)\to\R\Gamma_{\eet}(X,\Q_p(r))$$
  can be lifted canonically to the category of Banach-Colmez spaces; the $H^0$-term in the exact sequence (\ref{seq-intr}) is equal to the image of this period map.
  \item The exact sequence (\ref{seq-intr}) can be lifted canonically to the category of Banach-Colmez spaces; the $H^1$-term is the  identity component of $H^i{\mathbb R}\Gamma_{\synt}(X,r)$, the $H^0$-term --  the space of its connected components.
  \end{enumerate}
  \end{theorem}
  
  We also show that Theorem \ref{first} and Theorem \ref{second} have analogs for semistable formal schemes. 
 \begin{acknowledgments} We would like thank  Laurent Fargues for many very helpful  discussions concerning Banach-Colmez spaces and the Fargues-Fontaine curve.  Special thanks go to Pierre Colmez for explaining to us his work on Banach-Colmez spaces and helping us work out some of the examples included in this paper. We have also profited from exchanges with Bhargav Bhatt, Fr\'ed\'eric D\'eglise, Tony Scholl, and Peter Scholze.
 \end{acknowledgments}

  \subsubsection{Notation and conventions}
  For a field $L$, let ${\mathcal V}ar_{L}$ be the category of varieties over $L$.
    
  We will use a shorthand for certain homotopy limits. Namely,  if $f:C\to C'$ is a map  in the dg derived category of an abelian category, we set
$$[\xymatrix{C\ar[r]^f&C'}]:=\holim(C\to C^{\prime}\leftarrow 0).$$ 
We also set
$$
\left[\begin{aligned}
\xymatrix{C_1\ar[d]\ar[r]^f & C_2\ar[d]\\
C_3\ar[r]^g & C_4
}\end{aligned}\right]
:=[[C_1\stackrel{f}{\to} C_2]\to [C_3\stackrel{g}{\to} C_4]],
$$ 
where the diagram in the brackets  is a commutative diagram in the dg derived category.
\section{Preliminaries} Let $\so_K$ be a complete discrete valuation ring with fraction field
$K$  of characteristic 0 and with perfect
residue field $k$ of characteristic $p$. Let $\ovk$ be an algebraic closure of $K$ and let $\overline{\so}_K$ denote the integral closure of $\so_K$ in $\ovk$;  set $C:=\ovk^{\wedge}$ and let $C^{\flat}$ be its tilt with valuation $v$. Let
$W(k)=\so_F$ be the ring of Witt vectors of $k$ with 
 fraction field $F$ .  Set $G_K=\Gal(\overline {K}/K)$, and 
let $\phi=\phi_{W(\overline{k})}$ be the absolute
Frobenius on $W(\overline {k})$.

 In this section we will briefly recall some facts from $p$-adic Hodge Theory that may not yet be classical.

\subsection{Finite dimensional Banach-Colmez spaces.}
  Recall~\cite{CB} that a finite dimensional Banach-Colmez  space $W$ is, 
morally, a finite dimensional $C$-vector space 
 up to  a finite dimensional $\Q_p$-vector space. It has a {\em Dimension}\footnote{In~\cite{CB},
the dimension is called the ``dimension principale'', noted $\dim_{\rm pr}$, and
the height is called the ``dimension r\'esiduelle'', noted $\dim_{\rm res}$, and
the Dimension is called simply the ``dimension''.}
  ${\rm Dim}\,W=(a,b)$, where $a=\dim W\in\N$, the {\it dimension of $W$}, is the dimension of the $C$-vector space 
and $b={\rm ht}\,W\in\Z$, the {\it height of $W$}, is the dimension of the $\Q_p$-vector space. 
More precisely, a {\it Banach-Colmez space} ${\mathbb W}$ is a functor
$\Lambda\mapsto {\mathbb W}(\Lambda)$, from the category of
sympathetic algebras (spectral Banach algebras, such that $x\mapsto x^p$
is surjective on $\{x,\ |x-1|<1\}$; such an algebra is, in particular, perfectoid)
to the category of $\Q_p$-Banach spaces.
Trivial examples of such objects are:

$\bullet$ 
finite dimensional $\Q_p$-vector spaces $V$,
with associated functor $\Lambda\mapsto V$ for all $\Lambda$,

$\bullet$ ${\mathbb V}^d$, for $d\in\N$, with ${\mathbb V}^d(\Lambda)=\Lambda^d$, for
all $\Lambda$.

A Banach-Colmez  space ${\mathbb W}$ is {\it finite dimensional}
 if it ``is equal to ${\mathbb V}^d$, for some $d\in\N$, up to
finite dimensional $\Q_p$-vector spaces''.
More precisely, we ask that there exists finite dimensional $\Q_p$-vector spaces
$V_1,V_2$ and exact sequences
$$0\to V_1\to {\mathbb Y}\to {\mathbb V}^d\to 0,\quad
0\to V_2\to {\mathbb Y}\to {\mathbb W}\to 0,$$
so that ${\mathbb W}$ is obtained from ${\mathbb V}^d$ by ``adding $V_1$ and
moding out by $V_2$''.
Then $\dim{\mathbb W}=d$ and ${\rm ht}\,{\mathbb W}=\dim_{\Q_p}V_1-
\dim_{\Q_p}V_2$. (We are, in general, only interested in ${\mathbb W}(C)$ but, without
the extra structure, it would be impossible to speak of its Dimension.)
\begin{proposition}\label{BS1}
{\rm (i)} The Dimension of a finite dimensional Banach-Colmez  space is independant
of the choices made in its definition.

{\rm (ii)}  If $f:{\mathbb W}_1\to{\mathbb W}_2$ is a morphism of
finite dimensional Banach-Colmez  spaces, then ${\rm Ker}\,f$, ${\rm Coker}\,f$, and ${\rm Im}\,f$
are finite dimensional Banach-Colmez spaces, and
we have 
$$\dim {\mathbb W}_1=\dim {\rm Ker}\,f+\dim{\rm Im}\,f
\quad{\rm and}\quad
\dim {\mathbb W}_1=\dim {\rm Coker}\,f+\dim{\rm Im}\,f.$$

{\rm (iii)} If $\dim{\mathbb W}=0$, then ${\rm ht}\,{\mathbb W}\geq 0$.

{\rm (iv)} If ${\mathbb W}$ has an increasing filtration such that the successive quotients
are ${\mathbb V}^1$, and if ${\mathbb W}'$ is a sub-Banach-Colmez space of ${\mathbb W}$, then
${\rm ht}\,{\mathbb W}'\geq 0$.
\end{proposition}
\begin{proof}
The first two points are the core of the theory~\cite[Th.~0.4]{CB}.
The third point is obvious and the fourth is \cite[Lemme 2.6]{CF}.
\end{proof}

\subsection{Period rings}
\subsubsection{Period rings}
Main references for this section are \cite{FF0}, \cite{FF1}, \cite{LF1}.  
Let $\B^+_{\crr}$, $\B^+_{\st}$, $\B_{\dr}^+$,   $\B_{\crr}$, $\B_{\st}$, $\B_{\dr}$ be the Fontaine's rings of crystalline, semistable,  and de Rham periods, respectively. Let $\iota: \B_{\st}\hookrightarrow \B_{\dr}$ be the canonical embedding.

Let $\se=W(C^{\flat})[1/p]$. Define
\begin{align*}
 \B^{b} & =\{\sum_{n\gg-\infty}[x_n]p^n\in \se|\,\exists\, K ,\forall n \,\lvert x_n\rvert \leq K\},\\
 \B^{b,+} & =\{\sum_{n\gg-\infty}[x_n]p^n\in \se| x_n\in \so_{C^{\flat}}\}=W(\so_{C^{\flat}})[1/p].
 \end{align*}
For $x=\sum_n[x_n]p^n\in\B^b$ and $\rho\in ]0,1[$, $r\geq 0$, set 
$$\lvert x\rvert_{\rho}=\sup_{n}\lvert x_n\rvert \rho^n,\quad 
v_r(x)=\inf_{n\in{\mathbf Z}}\{v(x_n)+nr\}.
$$
If $\rho=p^{-r}\in]0,1[$, we have $\lvert x\rvert_{\rho}=p^{-v_r(x)}$. For $r\geq 0$, $v_r$ is a valuation on $\B^b$; thus, for $\rho\in]0,1]$, $\lvert\cdot\rvert_{\rho}$ is a multiplicative norm.
One defines 
the rings $\B$ and  $\B^+$ as the completions of $\B^b$ and $\B^{b,+}$, respectively,  with respect to $(\lvert\cdot\rvert _{\rho})_{\rho\in ]0,1[}$. 
For a compact interval $I\subset]0,1[$, the ring $\B_I$ is the completion of $\B^b$ with respect to $(\lvert\cdot\rvert _{\rho})_{\rho\in I}$. It is a PID.
The rings $\B$, $\B^+$ are $\Q_p$-Frechet algebras and $\B^+$ is the closure of $\B^{b,+}$ in $\B$. The ring $\B_I$ is a $\Q_p$-Banach algebra and
 we have $\B=\invlim_{I\in ]0,1[}\B_I$.

  We have 
  \begin{enumerate}
  \item $\B^+=\bigcap_{n\geq 0}\phi^n(\B^+_{\crr}).$\footnote{Hence $\B^+$ is the ring $\B^+_{\rig}$ from classical $p$-adic Hodge Theory.}
  \item  For an $F$-isocrystal $D$, $(D\otimes_{F}\B^+_{\crr})^{\phi=1}=(D\otimes_{F}\B^+)^{\phi=1}=(D\otimes_{F}\B)^{\phi=1}$. 
  \item For $d<0$, $\B^{\phi=p^d}=0$. $\B^{\phi=1}=\Q_p$, and, for $d\geq 0$, $\B^{\phi=p^d}=(\B^+)^{\phi=p^d}$.
\item There is a natural map $\iota: \B\to\B^+_{\dr}$ compatible with the embedding $\B^+_{\crr}\to\B^+_{\dr}$.
\end{enumerate}

 Let $\B^+_{\log}$ be the period ring defined in the same way as $\B^+_{\st}$ but starting from $\B^+$ instead of from $\B^+_{\crr}$. We will denote by $\iota: 
 \B^+_{\log}\to\B^+_{\dr}$ the canonical imbedding. We have a canonical map 
  $ \B^+_{\log}\to \B^+_{\st}$ compatible with all the structures. For an $F$-isocrystal $D$, we have
 \begin{equation}
 \label{formula-passage}
 (D\otimes_{F}\B^+_{\st})^{\phi=1,N=0}=(D\otimes_{F}\B^+_{\log})^{\phi=1,N=0}. 
 \end{equation}

   The Robba ring $\srr$ is defined as  $\srr=\dirlim_{\rho \to 0}\B_{]0,\rho]}$, where $\B_{]0,\rho]}:=\invlim_{0<\rho^{\prime} \leq \rho}\B_{[\rho^{\prime},\rho]}$ is the completion of $\B^b$ with respect to $(\lvert \cdot\rvert_{\rho^{\prime}})_{0<\rho^{\prime}\leq \rho}$. Since $\phi: \B_{]0,\rho]}\stackrel{\sim}{\to}\B_{]0,\rho^p]}$ the ring $\srr$ is equipped with a bijective Frobenius. By \cite[Theorem 2.9.6]{Ke}, the ring $\B_{]0,\rho]}$ is B\'ezout. Any closed ideal of $\B_{]0,\rho]}$ is principal. Hence the ring $\srr$ is B\'ezout as well.
   
   \subsubsection{Period Rings and some Banach-Colmez spaces}
   \label{kwak2}
   Recall \cite{CB}, that the above rings of periods can be also defined starting from any sympathetic algebra instead of $C$. One obtains Rings (of periods)
   ${\mathbb B}^+,   {\mathbb B}^+_{\st}, {\mathbb B}^+_{\dr}   $
   and  natural morphisms $\iota:  {\mathbb B}^+_{\st}\hookrightarrow  {\mathbb B}^+_{\dr}$, $\iota:  {\mathbb B}^+\hookrightarrow  {\mathbb B}^+_{\dr}$. 
   We have $\B^+={\mathbb B}^+(C)$, $\B^+_{\st}={\mathbb B}^+_{\st}(C)$, $\B^+_{\dr}={\mathbb B}^+_{\dr}(C) $. 
   
   The category $\BC$ of finite dimensional Banach-Colmez spaces is abelian. The functor of $C$-points
   $$
   \BC \to \Banach,\quad X\mapsto X(C).
   $$
   is exact and faithful.  In particular, if $C\kr\in C^b(\BC)$ is a complex of finite dimensional Banach-Colmez spaces then we have that its cohomology $H^i(C\kr)$ is a Banach-Colmez space as well and, for its $C$-points,  we have $H^i(C\kr)(C)=H^i(C\kr(C))$.

   
    Recall that a $\phi$-module  (over $F$) is a finite rank vector space over $F$ equipped with a bijective semilinear Frobenius $\phi: D\to D$. A  $(\phi,N)$-module (over $F$) is a finite $\phi$-module (over $F$ ) equipped with a monodromy operator $N: D\to D$ such that $N\phi=p\phi N$. A  filtered $(\phi,N)$-module (over $K$) is a finite $(\phi,N)$-module over $F$ such that $D_K:=D\otimes_F K$ is a filtered $K$-vector space. 
   
     To $D$, a finite filtered $(\phi,N)$-module over $K$,  and to $r\geq 0$, one can associate Banach-Colmez spaces
     \begin{align*}
     D\mapsto {\mathbb X}^r_{\st}(D) & =(D\otimes_Ft^{-r}{\mathbb B}^+_{\st})^{\phi=1,N=0}=(D\otimes_F{\mathbb B}^+_{\st})^{\phi=p^r,N=0},\\
       D\mapsto {\mathbb X}^r_{\dr}(D) & =(D_K\otimes_Kt^{-r}{\mathbb B}^+_{\dr})/F^0=(D_K\otimes_K{\mathbb B}^+_{\dr})/F^r.
           \end{align*}
          These are exact functors. We also have a natural transformation $ \iota:{\mathbb X}^r_{\st}(D)\to {\mathbb X}^r_{\dr}(D)$ induced by the morphism $\iota: {\mathbb B}^+_{\st}\hookrightarrow  {\mathbb B}^+_{\dr}$. 
          
           A filtered $(\phi,N,G_K)$-module is a  tuple $(D,\phi,N,\rho, F^{\scriptscriptstyle\bullet})$, where
\begin{enumerate}
\item $D$ is a  finite dimensional $F^{\nr}$-vector
space;
\item  $\phi : D \to D$ is a bijective semilinear  Frobenius  map;
\item $N : D \to D$ is  a $F^{\nr}$-linear monodromy map such that $N\phi = p\phi N$;
\item  $\rho$ is a $F^{\nr}$-semilinear $G_K$-action on $D$  (hence $\rho|I_K$ is linear) that factors through a finite quotient of the inertia $I_K$ and that commutes with $\phi$ and $N$;
 \item  $F^{\scriptscriptstyle\bullet}$ is a decreasing finite filtration  of $D_K:=(D\otimes _{F^{\nr}}\ovk)^{G_K}$ by $K$-vector
spaces.
\end{enumerate}
 The above functors
extend to exact functors
 \begin{align*}
     D\mapsto {\mathbb X}^r_{\st}(D) & =(D\otimes_{F^{\nr}}t^{-r}{\mathbb B}^+_{\st})^{\phi=1,N=0}=(D\otimes_{F^{\nr}}{\mathbb B}^+_{\st})^{\phi=p^r,N=0},\\
       D\mapsto {\mathbb X}^r_{\dr}(D) & =(D_K\otimes_Kt^{-r}{\mathbb B}^+_{\dr})/F^0=(D_K\otimes_K{\mathbb B}^+_{\dr})/F^r.
           \end{align*}
from filtered $(\phi,N,G_K)$-modules to finite dimensional Banach-Colmez spaces.           We also have a natural transformation $\iota:{\mathbb X}^r_{\st}(D)\to {\mathbb X}^r_{\dr}(D)$.

    Recall that we have (cf. \cite[Prop.~10.6]{CB}),
\begin{align*}
\dim {\mathbb X}^r_{K}(D)=&\ (r\dim_KD_{K}-\sum_{i=1}^r\dim F^iD_{K},0),\\
\dim {\mathbb X}^r_{\rm st}(D)=& \sum_{r_i\leq r}(r-r_i,1), \quad{\text{where the $r_i$'s are the
slopes of $\varphi$, repeated with multiplicity.}}
\end{align*}
In particular, if $F^{r+1}D_{K}=0$ and if all $r_i$'s are $\leq r$ (we let $r(D)$ be the smallest $r$
with these properties), then
$$\dim {\mathbb X}^r_{\rm st}(D)=
(r\dim_{F^{\nr}}D-t_N(D_{}),\dim_{F^{\nr}}D_{})
\quad{\rm and}\quad
\dim {\mathbb X}^r_{\rm dR}(D)=(r\dim_KD_{K}-t_H(D_{K}),0).$$
Here $t_N(D)=v_p(\det\phi)$ and $t_H(D)=\sum_{i\geq 0}i\dim_K(F^iD_{K}/F^{i+1}D_{K})$. 

The kernel of the map 
$\iota: {\mathbb X}^r_{\rm st}(D)\to X^r_{\rm dR}(D)$  is $V_{\pst}(D)$ if $r\geq r(D)$
(\cite[Prop.~10.14]{CB}), where $V_{\pst}(D):=(D\otimes_{F^{\nr}}\B_{\st})^{\phi=1,N=0}\cap F^0(D_K\otimes_K\B_{\dr})$.

\subsection{Vector bundles on the Fargues-Fontaine curve}
Main references for this section are \cite{FF0}, \cite{FF1}, \cite{LF1}.    
\subsubsection{Definitions}Let $X_{\ff}$ be the (algebraic) Fargues-Fontaine curve associated to the tilt $C^{\flat}$  and to  $\Q_p$. We have
$$
X_{\ff}=\Proj(P):= \Proj (\oplus_{d\geq 0}P_d),\quad P_d:=(\B^+_{\crr})^{\phi=p^d}.
$$
In the above definition we can replace $\B^+_{\crr}$ by $\B$ or $\B^+$.
The map $\theta: \B^+_{\crr}\to C$ determines a distinguished point ${\infty}\in X_{\ff}(C)$. We have the canonical identification $\B^+_{\dr}\stackrel{\sim}{\to} \widehat{\so}_{X,\infty}$; let $t$ be the uniformizing element of $\B_{\dr}^+$ ($t\in P_1-\{0\}$) so that  $\B_{\dr}=\B^+_{\dr}[1/t]$. 
$X_{\ff}$ is a regular noetherian scheme of Krull dimension one which is locally a spectrum of a Dedekind domain.

Let $\Bun_{X_{\ff}}$ denote the category of vector bundles on $X_{\ff}$. Since $X_{\ff}$ is locally a spectrum of a Dedekind domain $\Bun_{X_{\ff}}$  is a quasi-abelian category\footnote{An additive category with kernels and cokernels is called {\em quasi-abelian} if every pullback of a strict epimorphism is a strict epimorphism and every pushout of a strict monomorphism  is  a strict monomorphism. Equivalently, an additive category with kernels and cokernels is called {\em quasi-abelian} if  $\Ext(\cdot,\cdot)$ is bifunctorial.}  \cite[1.2.16]{An}.
It is thus equipped with the induced kernel-cokernel exact structure: a {\em short exact sequence} $$
0\to M\stackrel{f}{\to} N\stackrel{g}{\to} P\to 0
$$
is a pair of morphisms $(f,g)$ such that $M=\ker (g), P=\coker (f)$. This is the same as the natural exact structure: locally, use the embedding of the category of torsion free modules of finite rank  into its left abelian envelope - the category of finitely generated modules\footnote{Recall that a short sequence in a quasi-abelian category is exact if and only if it is exact in its left abelian envelope.}.

  For every $d\in\Z$,  there exists a line  bundle $\so(d):=\widetilde{P[d]}$. This is a line bundle of degree $d$ and every line bundle on the curve $X_{\ff}$ is isomorphic to some $\so(d)$.  One defines degree of a vector bundle as the degree of its determinant bundle; slope --  as degree divided by rank. 
  For every slope $\lambda\in\Q$, there exists a stable bundle $\so(\lambda)$ with slope $\lambda$. These vector bundles are constructed in the following way.    For $\lambda=d/h$,  $h\geq 1$, $(d,h)=1$, one takes the base change $X_{\ff,h}:=X_{\ff,\Q_{p^h}}$, where $\Q_{p^h}$ is the degree $h$ unramified extension of $\Q_p$, and defines  $\so(\lambda)=\so(d,h):=\pi_*\so(d)$, $\pi: X_{\ff,h}\to X_{\ff}$. We have 
 $$
 \so(d,h)=\widetilde{M(d,h)},\quad 
  M(d,h)=\bigoplus_{i\in\N}(\B^+)^{\phi^h=p^{ih+d}}.
  $$
  This is a vector bundle on $X_{\ff}$  of rank $h$ and degree $d$, hence of slope $\lambda$.  The global sections functor and the functor $V\mapsto V\otimes_{\Q_p}\so$  induce an equivalence of categories between semistable vector bundles of slope zero and finite dimensional $\Q_p$-vector spaces. Since semistable vector bundles of slope zero have vanishing $H^1$ this is an equivalence of exact categories. 
  \subsubsection{Cohomology of vector bundles}
  We have 
  \begin{align*}
  \so(\lambda)\otimes \so(\mu)\simeq \so(\lambda + \mu); & \quad \so(\lambda)^{\vee}=\so(-\lambda);\\
  \Hom(\so(\lambda),\so(\mu))=0, \quad \lambda > \mu; & \quad \Ext(\so(\lambda),\so(\mu))=H^1(X_{\ff},\so(\mu-\lambda))=0,  \quad  \lambda \leq  \mu.
   \end{align*}
\begin{example} \cite[12.1,12.3]{FF0} 

  We have 
  \begin{equation}
  \label{exampleB}  
  H^0(X_{\ff},\so(d))=\begin{cases}P_d & \mbox{if } d\geq 0,\\
  0 & \mbox{if } d<0.
  \end{cases}\quad 
H^1(X_{\ff},\so(d))=\begin{cases}
0 & \mbox{if } d\geq 0,\\
\B^+_{\dr}/(t^{-d}\B^+_{\dr}+\Q_p) & \mbox{if } d <0.
\end{cases} 
   \end{equation}
To obtain this we write  $X_{\ff}\setminus \{\infty\}=\Spec \B_e$, for $\B_e=\B_{\crr}^{\phi=1}=(\B^+[1/t])^{\phi=1}$ - a principal ideal domain. There is an equivalence of exact categories
$$\Bun_{X_{\ff}}\stackrel{\sim}{\to}\scc,\quad \se\mapsto (\Gamma(X_{\ff}\setminus\{\infty\},\se),\widehat{\se}_{\infty}).
$$
Here $\scc$ is the category of $\B$-pairs \cite{Ber}, i.e., the category of pairs  $(M,W)$, where $W$ is a free $\B^+_{\dr}$-module of finite type and $M\subset W[1/t]$ is a sub free $\B_e$-module of finite type such that $M\otimes_{\B_e}\B_{\dr}\stackrel{\sim}{\to} W[1/t]$. If $\se$ corresponds to the pair $(M,W)$, then its cohomology can be computed by the following \v{C}ech complex
$$
\R\Gamma(X_{\ff},\se)=(M\oplus W\stackrel{\partial}{\to} W[1/t]),\quad \partial(x,y)=x-y.
$$
 
The line bundle $\so(d)$ corresponds to the pair $(\B_e,t^{-d}\B^+_{\dr})$. Hence
$$
\R\Gamma(X_{\ff}, \so(d))=(\B_e\oplus t^{-d}\B^+_{\dr}\to\B_{\dr}).
$$
From this, since $\B_{\dr}=\B_e+\B_{\dr}^+$ and $\B_e\cap\B^+_{\dr}=\Q_p$, we get (\ref{exampleB}). Both $H^0(X_{\ff},\so(d))$ and $H^1(X_{\ff},\so(d))$ are $C$-points of finite dimensional Banach-Colmez spaces. These spaces have dimensions $(d,1)$ resp. $(0,0)$, for $d\geq 0$, and $(0,0)$ resp. $(-d,-1)$, for $d<0$. Hence the Euler characteristic $\chi(X_{\ff},\so(d))=(d,1)$. 

  More generally, for $\lambda=d/h$,  $h\geq 1$, $(d,h)=1$, we have 
   \begin{equation}
  \label{exampleB1}  
  H^0(X_{\ff},\so(d,h))=\begin{cases}(\B^+)^{\phi^h=p^d} & \mbox{if } d\geq 0,\\
  0 & \mbox{if } d<0.
  \end{cases}\quad 
H^1(X_{\ff},\so(d,h))=\begin{cases}
0 & \mbox{if } d\geq 0,\\
\B^+_{\dr}/(t^{-d}\B^+_{\dr}+\Q_{p^h}) & \mbox{if } d <0.
\end{cases} 
   \end{equation}
 The vector  bundle $\so(d,h)$ corresponds to the pair $((\B^+[1/t])^{\phi^h=p^d},(\B^+_{\dr})^h)$, with the glueing map
 $u: (\B^+[1/t])^{\phi^h=p^d}\to (\B_{\dr})^h$ defined as $x\mapsto (x,\phi(x),\ldots,\phi^{h-1}(x))$. Hence
$$
\R\Gamma(X_{\ff}, \so(d,h))=((\B^+[1/t])^{\phi^h=p^d}\oplus (\B^+_{\dr})^h\lomapr{u-\can}(\B_{\dr})^h)
$$
 and computation (\ref{exampleB1}) follows  since $\B_{\dr}=\B_e+\B_{\dr}^+$ and $(\B^+[1/t])^{\phi^h=1}\cap\B^+_{\dr}=\Q_{p^h}$.

   Again, both $H^0(X_{\ff},\so(d,h))$ and $H^1(X_{\ff},\so(d,h))$ are $C$-points of finite dimensional Banach-Colmez spaces. These spaces have dimensions $(d,h)$ resp. $(0,0)$, for $d\geq 0$, and $(0,0)$ resp. $(-d,-h)$, for $d<0$. Hence the Euler characteristic $\chi(X_{\ff},\so(d,h))=(d,h)$. 
      \end{example}
      \subsubsection{Classification of vector bundles}
 We have the following classification theorem for vector bundles on $X_{\ff}$. 
\begin{theorem}(Fargues-Fontaine \cite[Theorem 6.9]{FF1})
\begin{enumerate}
\item The semistable vector bundles of slope $\lambda$ on $X_{\ff}$ are the direct sums of $\so(\lambda)$.
\item The Harder-Narasimhan filtration of a vector bundle on $X_{\ff}$ is split.
\item There is a bijection
\begin{align*}
\{\lambda_1\geq \cdots\geq \lambda_n | n\in\N,\lambda_i\in\Q\} & \stackrel{\sim}{\to}\Bun_{X_{\ff}}/\sim\\
(\lambda_1,\ldots,\lambda_n) & \mapsto [\oplus_{i=1}^{n}\so(\lambda_i)].
\end{align*}
\end{enumerate}
\end{theorem}

  Let $ \Mod_{\srr}(\phi)$ be the category of $\phi$-modules over the Robba ring $\srr$, i.e., finite type projective $\srr$-modules $D$\footnote{Since $\srr$ is B\'ezout, an $\srr$-module $D$ is projective of finite type if and only if it is torsion free of finite type if and only if it is free of finite type.}  equipped with a $\phi$-linear isomorphism $\phi: D\stackrel{\sim}{\to} D$. Since $\srr$ is B\'ezout, this implies \cite[2.7.2]{An} that $ \Mod_{\srr}(\phi)$ is quasi-abelian. The induced kernel-cokernel exact structure is the same as the natural exact structure: use the embedding of the category $\Mod_{\srr}$ into its left abelian envelope - the category of finitely generated $\srr$-modules.

   Let $\Mod_{\B}(\phi)$ be the category of finite type projective  $\B$-modules $D$  equipped with a semi-linear  isomorphism $\phi: D{\to} D$. It is an exact category. We have the following diagram of equivalences of exact categories  \cite[Prop. 7.16, Theorem 7.18]{FF1}
\begin{equation}
\label{diag1}
\xymatrix{
 \Mod_{\srr}(\phi)
  &  \Mod_{\B}(\phi)\ar[l]_{\sim}\ar[r]^{\sim}& \Bun_{X_{\ff}}\\
 }
 \end{equation}
The first map is  induced by the inclusion $\B\subset \srr$. Via this map, the classification theorem of Kedlaya for $\phi$-modules over $\srr$ \cite{Ke} yields that there is a bijection 
\begin{align}
\label{classB}
\{\lambda_1\geq \cdots\geq \lambda_n | n\in\N,\lambda_i\in\Q\} & \stackrel{\sim}{\to}\Mod_{\B}(\phi)/\sim\\
(\lambda_1,\ldots,\lambda_n) & \mapsto [\oplus_{i=1}^{n}\B(-\lambda_i)].\notag
\end{align}
In particular, every $\phi$-module over $\B$ is free of finite type\footnote{But, in general, projective $\B$-modules are not free \cite[11.4.1]{FF0}.}.

The second map  in (\ref{diag1}) is defined by sending
\begin{equation}
D\mapsto \se(D),
\end{equation}
where $\se(D)$ is the sheaf associated to the $P$-graded module
$
\oplus_{d\geq 0}D^{\phi=p^d}.
$ Hence $\se(\B(i))=\so(-i)$ and 
we have $D\otimes\B^+_{\dr}\stackrel{\sim}{\to}\widehat{\se(D)}_{\infty}$. 
This equivalence of exact categories implies that the category $ \Mod_{\B}(\phi)$ is  also quasi-abelian\footnote{We note here that the ring $\B$ is not B\'ezout.}. Since the canonical exact structure on $\Bun_{X_{\ff}}$ agrees with the quasi-abelian kernel-cokernel exact structure it follows that this is also the case in $\Mod_{\B}(\phi)$. Since we know that this is also the case in $\Mod_{\srr}(\phi)$, it follows that the first map in (\ref{diag1}) is an equivalence of exact categories.

 \subsubsection{Modifications of vector bundles and  filtered $\phi$-modules} We will recall the definitions of these categories \cite{LF1}.
  A {\em modification} of vector bundles on the curve $X_{\ff}$ is a triple $(\se_1,\se_2,u)$, where $\se_1$, $\se_2$ are vector bundles on $X_{\ff}$ and $u$ is an isomorphism 
$\se_{1|X_{\ff}\setminus\{\infty\}}\simeq \se_{2|X_{\ff}\setminus\{\infty\}}$. 
Modification $(\se_1,\se_2,u)$ is
 called {\em effective} if $u(\se_1)\subset \se_2$; it is called 
 {\em admissible} if $\se_1$ is a semistable vector bundle of slope $0$. Let $\Modif$, $\Modif^{+}$, $\Modif^{\ad}$ denote the corresponding  exact categories. 

   Modifications can be also described as pairs $(\se, \Lambda)$, where $\se$ is a vector bundle and $\Lambda$ is a $\B^+_{\dr}$-lattice $\Lambda\subset \widehat{\se}_{\infty}[1/t]$.  
To a modification $(\se_1,\se_2,u)$ one associates the pair $(\se_2,u(\widehat{\se}_{1,\infty}))$. Effective modifications correspond to pairs $(\se,\Lambda)$ such 
that $\Lambda\subset \widehat{\se}_{\infty}$. This correspondence preserves exact structures.

   A filtered $\phi$-module (over $\ovk$) \footnote{Fargues' \cite[4.2.2]{LF1} original name was "$\phi$-module jaug\'e".}  $(D,\Lambda)$  consists of a $\phi$-module $D$ over $\B^+$ and a $\B^+_{\dr}$-lattice $\Lambda\subset D\otimes \B_{\dr}$. It is {\em effective} if 
$\Lambda\subset D\otimes \B^+_{\dr}$. A filtered $\phi$-module  is called {\em admissible} if 
$
D^{\phi=1}\cap\Lambda
$ is a $\Q_p$-vector space of rank equal to the rank of $D$ over $\B^+$. Denote by 
$
\G, \G^{+}, \G^{\ad}
$
the corresponding exact categories. 
We have an equivalence of categories $\G\stackrel{\sim}{\to}\Modif$ that induces equivalences between the effective and admissible subcategories. We note that this is not an equivalence of exact categories.
\section{Extensions of modifications}In this section we study extensions in the categories of modifications and of filtered $\phi$-modules over $\ovk$.
\subsection{Extensions of $\phi$-modules}We start with extensions of $\phi$-modules. 
 Let $\Mod_{\B^+}$ be the category of free $\B^+$-modules of finite rank. It is an exact category (with a split exact structure) and we will denote by $\sd_{\B^+}$, $D_{\B^+}$ its derived dg category and derived category, respectively.  We note that, since the exact structure is split, all quasi-isomorphisms are actually homotopy equivalences. In particular, 
 $D_{\B^+}=H_{\B^+}$ (the homotopy category).
 For $D_1,D_2\in \sd^b_{\B^+}$, we have a quasi-isomorphism  $\Hom_{\sd^b_{\B^+}}(D_1,D_2)\stackrel{\sim}{\to }\Hom_{\B^+}(D_1,D_2)$. We have  similar statements for analogous categories $\Mod_{\B^+_{\dr}}$ and $\Mod_{\B_{\dr}}$. We note that, since $\B^+_{\dr}$ is a PID and $\B_{\dr}$ is a field, these two categories are quasi-abelian (with the kernel-cokernel exact structure equal to the natural one) \cite[2.7.2]{An}. In $\Mod_{\B^+_{\dr}}$, a morphism is strict if and only if its cokernel taken in the category of $\B^+_{\dr}$-modules is torsion-free or, equivalently, its kernel is $t$-saturated in the ambient module\footnote{A $\B^+_{\dr}$-module $N$  is called $t$-saturated in a $\B^+_{\dr}$-module $M$, for $N\hookrightarrow M$,  if every $x\in M$ such that $tx\in N$ is actually in $N$.}.

Let $\Mod_{\B^+}(\phi)$ be the category of free $\B^+$-modules $D$  of finite rank equipped with an isomorphism $\phi^*D\stackrel{\sim}{\to} D$. It is an exact category.
The exact structure is split, i.e., $\Ext^1$ in $\Mod_{\B^+}(\phi)$ is trivial. To see this, one first proves a classification of $\phi$-modules over $\B^+$ analogous to the one for $\B$ in (\ref{classB})  then computes $\Ext^1$ for the simple modules \cite[Theorem 7.23, Proposition 7.25]{LF1}. 

   Let $\Mod_{F^{\nr}}(\phi)$ be the category of finite rank modules over $F^{\nr}$ with a semilinear isomorphism $\phi: D\to D$. It is an exact category with split exact structure. There is an (exact)  functor $\Mod_{F^{\nr}}(\phi)\to \Mod_{\B^+}(\phi)$, $D\mapsto D\otimes_{K^{\nr}_0}\B^+$. Using the Dieudonn\'e-Manin decomposition in $\Mod_{F^{\nr}}(\phi)$ we see that this functor induces bijection on objects of the two categories. Clearly though we have a lot more morphisms in $\B^+$-modules.

   We will denote by $\sd_{\B^+}(\phi)$, $D_{\B^+}(\phi)$  the  derived  categories of $\Mod_{\B^+}(\phi)$. For $D_1,D_2\in \Mod_{\B^+}(\phi)$, let $\Hom_{\B^+,\phi}(D_1,D_2)$ denote the group of Frobenius morphisms. We have the exact sequence
\begin{align}
\label{exact}
0\to \Hom_{\B^+,\phi}(D_1,D_2)\to \Hom_{\B^+}(D_1,D_2)\stackrel{\delta}{\to } \Hom_{\B^+}(D_1,\phi_*D_2)\to 0.\end{align}
where  $\delta: x\mapsto \phi_{D_2}x-\phi_*(x)\phi_{D_1}$. Hence,  $\Hom_{\B^+,\phi}(D_1,D_2)=\Cone(\delta)[-1]$.
It follows that, for $D_1,D_2\in \sd^b_{\B^+}(\phi)$, we have a quasi-isomorphism  $\Hom_{\sd^b_{\B^+}(\phi)}(D_1,D_2)\stackrel{\sim}{\to }\Hom_{\B^+,\phi}(D_1,D_2)$, i.e., 
$$
\Hom_{\sd^b_{\B^+}(\phi)}(D_1,D_2)=\Cone(\Hom_{\B^+}(D_1,D_2)\stackrel{\delta}{\to} \Hom_{\B^+}(D_1,\phi_*D_2))[-1]
$$

   We compute similarly that (with the obvious notation),  for $D_1,D_2\in \sd^b_{\B}(\phi)$, we have a quasi-isomorphism
   $$
   \Hom_{\sd^b_{\B}(\phi)}(D_1,D_2)=\Hom_{\B,\phi}(D_1,D_2):=\Cone(\Hom_{\B}(D_1,D_2)\stackrel{\delta}{\to} \Hom_{\B}(D_1,\phi_*D_2))[-1]
   $$
The inclusion $\B^+\subset \B$  defines an equivalence of categories \cite[7.7]{FF1}
\begin{equation}
\label{exact1}
\Mod_{\B^+}(\phi)\stackrel{\sim}{\to} \Mod_{\B}(\phi).
\end{equation}
This is shown in \cite{FF1} by proving that the above functor is fully faithful \cite[Prop. 7.20]{FF1}, what implies, by the classification of $\phi$-modules over $\B$ (\ref{classB}), an analogous classification of $\phi$-modules over $\B^+$. The wanted equivalence of categories follows. 
On the other hand, if we equip  $\Mod_{\B^+}(\phi)$ with the natural exact structure, the functor (\ref{exact1}) is not an equivalence of exact categories. It is because for the natural exact structure
$\Ext(\B^+,\B^+(1))=0$ (by (\ref{exact})) but for the quasi-abelian kernel-cokernel exact structure
 $\Ext(\B^+,\B^+(1))=\Ext(\B,\B(1))=\Ext(\so,\so(-1))=H^1(X_{\ff},\so(-1))=C/\Qp$. The corresponding exact sequences have cokernel maps which are not surjective for the natural exact structure.  In what follows we will work exclusively with the natural exact structure.

 \subsection{Extensions of filtered $\phi$-modules}
 
   Consider the quasi-abelian 
category $\Fmod_{\B^+_{\dr}}$ of pairs $(\Lambda, M)$, $\Lambda\subset M$, where  $M\in \Mod_{\B^+_{\dr}}$, and $\Lambda$ is a $\B^+_{\dr}$-lattice in $M[1/t]$.   We note that a morphism $(f_{\Lambda},f_M): (\Lambda,M)\to (\Lambda^{\prime},M^{\prime})$ is strict \cite[1.1.3]{Sn} if and only if so are the morphisms $f_{\Lambda}$ and $f_M$. Since $\B^+_{\dr}$ is a PID, elementary divisors theory, gives us that every exact sequence
$$
0\to (\Lambda_1,M_1)\to (\Lambda_2,M_2)\to (\Lambda_3, M_3)\to 0$$ splits. Hence, for $M_1,M_2\in \sd^b(\Fmod_{\B^+_{\dr}})$, we have  
$\Hom_{\sd^b(\Fmod_{\B^+_{\dr}})}(M_1,M_2)=\Hom_{\Fmod_{\B^+_{\dr}}}(M_1,M_2)$.

Similarly, we define the   quasi-abelian category $\Fmod_{\B_{\dr}}$ of pairs $(\Lambda,M)$, $\Lambda\subset M$, where $M\in \Mod_{\B_{\dr}}$, and $\Lambda$ is a $\B^+_{\dr}$-lattice in $M$. Again, for $M_1,M_2\in \sd^b(\Fmod_{\B_{\dr}})$, we have
 $\Hom_{\sd^b(\Fmod_{\B_{\dr}})}(M_1,M_2)=\Hom_{\Fmod_{\B_{\dr}}}(M_1,M_2)$. 

 Let $M=(D,\Lambda)$, $T=(D^{\prime},\Lambda^{\prime})$ be two complexes in $C^b(\G^{+} )$, $C^b(\G )$, respectively. Define the complexes   $\Hom^{+}(M,T)$, $\Hom^{}(M,T)$ as the following homotopy fibers
\begin{align*}
 \Hom^{+}(M,T) & :=[\Hom_{\B^+,\phi}(D,D^{\prime})\oplus  \Hom_{F\Mod_{\B^+_{\dr}}}((\Lambda,D_{\B^+_{\dr}}),(\Lambda^{\prime},D^{\prime}_{\B^+_{\dr}}))\verylomapr{\can-\can}
 \Hom_{\B^+_{\dr}}(D_{\B^+_{\dr}},D^{\prime}_{\B^+_{\dr}} )],\\
  \Hom^{}(M,T) & :=[\Hom_{\B^+,\phi}(D,D^{\prime})\oplus  \Hom_{F\Mod_{\B_{\dr}}}((\Lambda,D_{\B_{\dr}}),(\Lambda^{\prime},D^{\prime}_{\B_{\dr}}))\verylomapr{\can-\can}
 \Hom_{\B_{\dr}}(D_{\B_{\dr}},D^{\prime}_{\B_{\dr}} )]
\end{align*}
Complexes $\Hom^{+}$, $\Hom^{}$ compose naturally.  
 \begin{proposition}
\label{comp1}
 We have ($*=\Vtextvisiblespace[0.2cm], \ad$)
 \begin{align*}
 \Hom_{\sd^b(\G^{+,*} )}(M,T)\simeq \Hom_{}^{+}(M,T), \quad  \Hom_{\sd^b(\G^* )}(M,T)\simeq \Hom^{}(M,T).
   \end{align*}
 \end{proposition}
 \begin{proof}
 Proof is analogous to the one of Proposition 2.7 in \cite{DN}: note that $\Cone(M\stackrel{\id}{\to} M)$, for $M\in \sd^b(\G^{+} )$, $M\in \sd^b(\G )$,   is acyclic and that the category of semistable vector bundles of slope zero is closed under extensions (in the category of vector bundles).
 \end{proof}
 
 Let $\un:=(\B^+,\B^+_{\dr})$ be the unit filtered $\phi$-module.  For $M=(D,\Lambda)$, $M\in\G^{+} $ and 
 $M\in\G $  we set $H^*_{+}(\ovk,M):=H^*\R\Hom^{+}(\un,M)$ and $H^*(\ovk,M):=H^*\R\Hom(\un,M)$, respectively. We have 
 \begin{equation}
 \label{form}
 H^i_{+}(\ovk,M)=
  \begin{cases}
 D^{\phi=1}\cap \Lambda, & i=0,\\
   D\otimes_{\B^+}\B^+_{\dr}/(\Lambda+D^{\phi=1}) & i=1,\\
  0 & i\geq 2.
  \end{cases}\quad H^i(\ovk,M) =
   \begin{cases}
 D^{\phi=1}\cap \Lambda_M & i=0,\\
  D\otimes_{\B^+}\B_{\dr}/(\Lambda+ D^{\phi=1}) & i=1,\\
  0   & i\geq 2.
  \end{cases}
 \end{equation}
  Moreover, the complex 
  $$\R\Hom^{+}(\un,M)  := (D^{\phi=1}\to D\otimes_{\B^+}\B^+_{\dr}/\Lambda)
    $$
  can be lifted canonically to a complex of finite dimensional Banach-Colmez spaces. To see that, set, for a sympathetic algebra $A$, 
  \begin{align}
  \label{bcspaces}
  D(A):=D\otimes_{\B^+}{\mathbb B}^+(A),\quad \Lambda(A):=\Lambda\otimes _{\B^+_{\dr}}{\mathbb B}^+_{\dr}(A),\\
  {\mathbb R}\Hom^{+}(\un,M)(A)  := (D(A)^{\phi=1}\to D\otimes_{\B^+}{\mathbb B}^+_{\dr}(A)/\Lambda(A)).\notag
    \end{align}
The bottom complex  is clearly a complex of Banach-Colmez spaces such that $  {\mathbb R}\Hom^{+}(\un,M)(C)={R}\Hom^{+}(\un,M)$. The fact that 
 this is a complex of finite dimensional Banach-Colmez spaces follows from section (\ref{kwak2}).
  
 \begin{remark}
   Let $(\se_1,\se_2,u)$ be an effective modification and let $M$ be the associated effective filtered $\phi$-module. We have an exact sequence
   of sheaves on $X_{\ff}$
   $$
   0\to \se_1\stackrel{u}{\to} \se_2\to i_{\infty*}(\widehat{\se}_{2,\infty}/u(\widehat{\se}_{1,\infty}))\to 0
   $$
   We get from it  the long exact sequence of cohomology groups
   \begin{align*}
   0\to  & H^0(X_{\ff},\se_1)\to H^0(X_{\ff},\se_2)\to \widehat{\se}_{2,\infty}/u(\widehat{\se}_{1,\infty})\\
    & \to H^1(X_{\ff}, \se_1)\to H^1(X_{\ff},\se_2)\to 0
    \end{align*}
   Since $H^0(X_{\ff},\se(D))=D^{\phi=1}$, by (\ref{form}), we have that 
   \begin{align*}
  H^0_{+}(\ovk,M) & =H^0(X_{\ff},\se_1),\\
H^1_{+}(\ovk,M) & =\ker(H^1(X_{\ff},\se_1)\to H^1(X_{\ff},\se_2)).
\end{align*} 
In particular, if $(\se_1,\se_2,u)$ is effective and admissible then $H^1_{+}(\ovk,M)=0$ because $H^1(X_{\ff},\se_1)=0$.
\end{remark}
 \begin{remark}
 We note that if $M=\un$, then $H^1_{+}(\ovk,M)=0$ and $H^1(\ovk,M)=\B_{\dr}/\B_{\dr}^+$; hence effective filtered $\phi$-modules are not closed under extensions in the category of filtered $\phi$-modules. On the other hand, admissible filtered $\phi$-modules are closed under extensions (because semistable vector bundles of slope zero are closed under extensions in the category of vector bundles).
  \end{remark}
\subsection{Extensions of modifications}Extensions of modifications of vector bundles can be computed in an analogous way; we  will just list the results. 
Let $M=(D,\Lambda)$, $T=(D^{\prime},\Lambda^{\prime})$ be two complexes in $C^b(\Modif^+)$, $C^b(\Modif)$, respectively, with $D, D^{\prime}$ --  complexes of $\phi$-modules over $\B$.  Define the respective complexes   $\Hom^{+}(M,T)$, $\Hom^{}(M,T)$ as the following homotopy fibers
\begin{align*}
 \Hom^{+}(M,T) & :=[\Hom_{\B,\phi}(D,D^{\prime})\oplus  \Hom_{F\Mod_{\B^+_{\dr}}}((\Lambda,D_{\B^+_{\dr}}),(\Lambda^{\prime},D^{\prime}_{\B^+_{\dr}}))\verylomapr{\can-\can}
 \Hom_{\B^+_{\dr}}(D_{\B^+_{\dr}},D^{\prime}_{\B^+_{\dr}} )],\\
  \Hom^{}(M,T) & :=[\Hom_{\B,\phi}(D,D^{\prime})\oplus  \Hom_{F\Mod_{\B_{\dr}}}((\Lambda,D_{\B_{\dr}}),(\Lambda^{\prime},D^{\prime}_{\B_{\dr}}))\verylomapr{\can-\can}
 \Hom_{\B_{\dr}}(D_{\B_{\dr}},D^{\prime}_{\B_{\dr}} )]
\end{align*}
Complexes $\Hom^{+}$, $\Hom^{}$ compose naturally.  
\begin{proposition}
 We have ($*=\Vtextvisiblespace[0.2cm], \ad$)
 \begin{align*}
 \Hom_{\sd^b(\Modif^{+,*} )}(M,T)\simeq \Hom_{}^{+}(M,T), \quad  \Hom_{\sd^b(\Modif^*)}(M,T)\simeq \Hom^{}(M,T).
   \end{align*}
 \end{proposition}
 Let $\un:=(\B,\B^+_{\dr})$ be the unit  modification.  For $M=(D,\Lambda)$, $M\in\Modif^+$ and 
 $M\in\Modif$,  we set $H^*_{+}(\Modif,M):=H^*\R\Hom^{+}(\un,M)$ and $H^*(\Modif,M):=H^*\R\Hom(\un,M)$, respectively. We have the following long exact sequence ($*=+, \Vtextvisiblespace[0.2cm]$)
\begin{align*}
0\to H^0_*(\Modif,M)\to D^{\phi=1}\to (D\otimes\B^*_{\dr})/\Lambda\to H^1_*(\Modif,M)\to D/(1-\phi)D\to 0
\end{align*}
Moreover, for $i\geq 2$, $H^i_*(\Modif,M)=0$.

  We conclude that, for an effective filtered $\phi$-module  $M=(D,\Lambda)$ over $\ovk$, we have $H^i_+(\ovk,M)\hookrightarrow H^i_+(\sm,M_{\B}),$ where $M_{\B}:=(D\otimes_{\B^+}\B,\Lambda)$. More specifically, we have
\begin{enumerate}
\item $H^0_+(\ovk,M)=H^0_+(\sm,M_{\B})$, 
\item $H^1_+(\ovk,M)\hookrightarrow H^1_+(\sm,M_{\B})$ with cokernel $H^1(X_{\ff},\se(D))=D/(1-\phi)D$,
\item  $H^i_+(\ovk,M)= H^i_+(\sm,M_{\B})=0$, for $i\geq 2$. 
\end{enumerate}

  The following proposition shows that cohomology of effective modifications  recovers only the cohomology of the "smaller" modified vector bundle on the Fargues-Fontaine curve.
\begin{proposition}
 Let $T=(\se_1,\se_2,u)$ be an effective modification. We have a canonical quasi-isomorphism
$$ \R\Hom_{\Modif^+}(\un,T)\simeq \R\Gamma(X_{\ff},\se_1).
$$
\end{proposition}
\begin{proof}
Let $T=(\se_1,\se_2,u)$ and $T^{\prime}=(\se_1^{\prime},\se_2^{\prime},u^{\prime})$ be two of complexes of effective modifications. We have
$$
\R\Hom_{\sm^+}(T,T^{\prime})=[\R\Hom_{\Bun_{X_{\ff}}}(\se_1,\se_1^{\prime})\oplus\R\Hom_{\Bun_{X_{\ff}}}(\se_2,\se_2^{\prime})
\lomapr{u^{\prime}_*-u^*} \R\Hom_{\Bun_{X_{\ff}}}(\se_1,\se_2^{\prime})]
$$
If we apply this to the unit modification $\un$ and $T=(\se_1,\se_2,u)$ we find that 
\begin{align*}
\R\Hom_{\sm^+}(\un,T) & =[\R\Hom_{\Bun_{X_{\ff}}}(\so,\se_1)\oplus\R\Hom_{\Bun_{X_{\ff}}}(\so,\se_2)\lomapr{u_*-\id} \R\Hom_{\Bun_{X_{\ff}}}(\so,\se_2)]\\
 & =\R\Hom_{\Bun_{X_{\ff}}}(\so,\se_1)=\R\Gamma(X_{\ff},\se_1),
\end{align*}
as wanted.
\end{proof}
\section{Geometric syntomic cohomology}
We will recall the definition of geometric syntomic cohomology defined in \cite{NN} and list its basic properties. 
\subsection{Definitions}\label{kwak?}
   For $X\in \mathcal{V}ar_{\ovk}$, we have the rational crystalline cohomology  $\R\Gamma_{\crr}(X)$ defined in \cite{BE2} using $h$-topology. It is a filtered dg  perfect $\B^+_{\crr}$-algebra equipped with the Frobenius action $\phi$. The Galois group $G_K$ acts on ${\mathcal V}ar_{\ovk}$ and it acts on $X\mapsto \R\Gamma_{\crr}(X)$ by transport of structure. If $X$ is defined over $K$ then $G_K$ acts naturally on $\R\Gamma_{\crr}(X)$. 
   
   For $r\geq 0$, one defines \cite{NN} geometric syntomic cohomology of $X$ as the $r$'th filtered Frobenius eigenspace of crystalline cohomology
   $$
   \R\Gamma_{\synt}(X,r):=[F^r\R\Gamma_{\crr}(X)\lomapr{1-\phi_r}\R\Gamma_{\crr}(X)].
   $$ In the case when $X$ is the generic fiber of a proper semistable scheme $\sx$ over $\so_K$, this agrees with the (continuous) logarithmic syntomic cohomology of Fontaine-Messing-Kato.

   The above definition is convenient to study period maps but for computations a different definition is more convenient. We will explain it now. 
   There is a natural map $\gamma: \R\Gamma_{\crr}(X)\to \R\Gamma_{\dr}(X)\otimes_{\ovk}\B^+_{\dr}$. Here $\R\Gamma_{\dr}(X)$ is the Deligne's de Rham cohomology.
   It is a filtered perfect complex of $\ovk$-vector spaces. We equip it with the Hodge-Deligne filtration. It follows from the degeneration of the Hodge-de Rham spectral sequence that the differentials in $\R\Gamma_{\dr}(X)$ are strict for the filtration (cf. \cite[Prop. 8.3.1]{HMF}). 
   
   The complex $\R\Gamma_{\dr}(X)\otimes_{\ovk}\B^+_{\dr}$ is a perfect complex of free $\B^+_{\dr}$-modules. It is  filtered by a perfect complex of $\B^+_{\dr}$-lattices. The cohomology of $F^i(\R\Gamma_{\dr}(X)\otimes_{\ovk}\B^+_{\dr})$, $i\geq 0$, is torsion-free:
   this follows from the degeneration of the Hodge-de Rham spectral sequence. Moreover, this implies that this complex is strict: kernels of differentials are $t$-saturated because the complex is perfect, images are $t$-saturated in the kernels because cohomology is torsion-free; this implies that the images are $t$-saturated in the ambient modules, as wanted.

   For $r\geq 0$, the geometric syntomic cohomology of $X$ can be defined in the following way.
 \begin{align*}
\R\Gamma_{\synt}(X,r) & :=
  \left[\xymatrix@C=36pt{\R\Gamma_{\crr}(X)\ar[rr]^-{(1-\phi_r,\gamma)} & & \R\Gamma_{\crr}(X)\oplus (\R\Gamma_{\dr}(X)\otimes_{\ovk}\B^+_{\dr}) /F^r}\right]\\
 &   \stackrel{\sim}{\leftarrow}
   \left[\begin{aligned}\xymatrix@C=50pt{\R\Gamma_{\hk}(X)\otimes_{F^{\nr}}\B_{\st}^+\ar[r]^-{(1-\phi_r,\iota_{\dr}\otimes\iota)}\ar[d]^{N}  & \R\Gamma_{\hk}(X)\otimes_{F^{\nr}}\B_{\st}^+\oplus (\R\Gamma_{\dr}(X)\otimes_{\ovk}\B^+_{\dr}) /F^r\ar[d]^{(N,0)}\\
\R\Gamma_{\hk}(X)\otimes_{F^{\nr}}\B_{\st}^+\ar[r]^{1-\phi_{r-1}}  & \R\Gamma_{\hk}(X)\otimes_{F^{\nr}}\B_{\st}^+}\end{aligned}\right]
 \end{align*}
 We set $\phi_i:=\phi/p^i$. Here $\R\Gamma_{\hk}(X)$ is the Beilinson's Hyodo-Kato cohomology of $X$ \cite{BE2}. It is a complex of finite rank $(\phi,N)$-modules over $F^{\nr}$. It comes equipped with the Hyodo-Kato quasi-isomorphism $$\iota_{\dr}:\R\Gamma_{\hk}(X)\otimes_{F^{\nr}}\ovk\simeq \R\Gamma_{\dr}(X).$$ The second quasi-isomorphism in the above diagram
 uses  the quasi-isomorphism $$
 \iota_{\crr}: \R\Gamma_{\hk}(X)\otimes_{F^{\nr}}\B_{\st}^+\stackrel{\sim}{\to}\R\Gamma_{\crr}(X)\otimes_{\B^+_{\crr}}\B^+_{\st}
 $$ 
 that is compatible with the action of  $\phi$ and $N$. We will write
 $\R\Gamma_{\hk}(X)^{\tau}_{\B^+_{\crr}}:=(\R\Gamma_{\hk}(X)\otimes_{F^{\nr}}\B_{\st}^+)^{N=0}$. We have a  trivialization
 \begin{align*}
 \R\Gamma_{\hk}(X)\otimes_{F^{\nr}}\B_{\crr}^+ & \stackrel{\sim}{\to} \R\Gamma_{\hk}(X)^{\tau}_{\B^+_{\crr}}=(\R\Gamma_{\hk}(X)\otimes_{F^{\nr}}\B_{\st}^+)^{N=0},\\
&  x \mapsto \exp(N(x)\log([\tilde{p}])),
 \end{align*}
 where $\tilde{p}$ is a sequence of $p^n$'th roots of $p$.
  This yields a quasi-isomorphism $\R\Gamma_{\hk}(X)^{\tau}_{\B^+_{\crr}}\simeq\R\Gamma_{\crr}(X)$. Both maps are compatible with Frobenius and monodromy.
   
   We can rewrite the above in the following form
   \begin{align*}
\R\Gamma_{\synt}(X,r)
 & \stackrel{\sim}{\to}
  [\xymatrix{\R\Gamma_{\crr}(X)^{\phi=p^r}\ar[rr]^-{\gamma} & &  (\R\Gamma_{\dr}(X)\otimes_{\ovk}\B^+_{\dr}) /F^r}]\notag\\
&   \stackrel{\sim}{\leftarrow}
   [\xymatrix{(\R\Gamma_{\hk}(X)\otimes_{F^{\nr}}\B_{\st}^+)^{\phi=p^r,N=0}\ar[r]^-{\iota_{\dr}\otimes\iota}  & (\R\Gamma_{\dr}(X)\otimes_{\ovk}\B^+_{\dr}) /F^r}],
 \end{align*}
 where we set $$\R\Gamma_{\crr}(X)^{\phi=p^r}:=[\R\Gamma_{\crr}(X)\lomapr{p^r-\phi} \R\Gamma_{\crr}(X)]$$
and
 $$
 (\R\Gamma_{\hk}(X)\otimes_{F^{\nr}}\B_{\st}^+)^{\phi=p^r,N=0}:=\left[\begin{aligned}\xymatrix{
 \R\Gamma_{\hk}(X)\otimes_{F^{\nr}}\B_{\st}^+  \ar[r]^{1-\phi_r}\ar[d]^N    & \R\Gamma_{\hk}(X)\otimes_{F^{\nr}}\B_{\st}^+\ar[d]^N\\ 
     \R\Gamma_{\hk}(X)\otimes_{F^{\nr}}\B_{\st}^+ \ar[r]^{1-\phi_{r-1}}     & \R\Gamma_{\hk}(X)\otimes_{F^{\nr}}\B_{\st}^+  }
    \end{aligned}\right]
    $$
    Alternatively, by formula (\ref{formula-passage}),  we can change the period ring $\B^+_{\st}$ to $\B^+_{\log}$ to obtain
     \begin{align*}
\R\Gamma_{\synt}(X,r)
  \stackrel{\sim}{\to}
   [\xymatrix{(\R\Gamma_{\hk}(X)\otimes_{F^{\nr}}\B_{\log}^+)^{\phi=p^r,N=0}\ar[r]^-{\iota_{\dr}\otimes\iota}  & (\R\Gamma_{\dr}(X)\otimes_{\ovk}\B^+_{\dr}) /F^r}].
 \end{align*}
 We will write
 $\R\Gamma_{\hk}(X)^{\tau}_{\B^+}:=(\R\Gamma_{\hk}(X)\otimes_{F^{\nr}}\B_{\log}^+)^{N=0}$; we have a canonical trivialization
 $\R\Gamma_{\hk}(X)^{\tau}_{\B^+}\simeq \R\Gamma_{\hk}(X)\otimes_{F^{\nr}}\B_{}^+$ (compatible with Frobenius and monodromy).  With this notation, we have
  \begin{align}
  \label{analog1}
\R\Gamma_{\synt}(X,r)
  \stackrel{\sim}{\to}
   [\xymatrix{\R\Gamma_{\hk}(X)_{\B^+}^{\tau,\phi=p^r}\ar[r] & (\R\Gamma_{\dr}(X)\otimes_{\ovk}\B^+_{\dr}) /F^r}].
 \end{align}
   \subsection{Basic properties}We will now list basic properties of geometric syntomic cohomology. 
   \subsubsection{Syntomic period maps}Let $X\in {\mathcal V}ar_{\ovk}$.
Recall that Beilinson \cite{Be1}, \cite{BE2} defined comparison quasi-isomorphisms
\begin{align*}
& \rho_{\crr}:  \R\Gamma_{\crr}(X)\otimes_{\B^+_{\crr}}\B_{\crr}\simeq\R\Gamma_{\eet}(X,\Qp)\otimes_{\Qp}\B_{\crr},\quad & \rho_{\hk}:    \R\Gamma_{\hk}(X)\otimes_{F^{\nr}}\B_{\st}\simeq\R\Gamma_{\eet}(X,\Qp)\otimes_{\Qp}\B_{\st},\\
& \rho_{\dr}:   \R\Gamma_{\dr}(X)\otimes_{\ovk}\B_{\dr}\simeq\R\Gamma_{\eet}(X,\Qp)\otimes_{\Qp}\B_{\dr}
\end{align*}
that are compatible with the extra structures and with each other.
For $r\geq 0$, they give us the syntomic  period map \cite{NN} 
$$\rho_{\synt}: \R\Gamma_{\synt}(X,r)\to \R\Gamma_{\eet}(X,\Qp(r))$$
defined as follows
\begin{align}
\label{quasi}
\R\Gamma_{\synt}(X,r) & \simeq [(\R\Gamma_{\hk}(X)\otimes_{F^{\nr}}\B^+_{\st})^{\phi=p^r,N=0}
\verylomapr{\iota_{\dr}\otimes\iota}(\R\Gamma_{\dr}(X)\otimes_{\ovk}\B^+_{\dr})/F^r]\\
 & \to  [(\R\Gamma_{\hk}(X)\otimes_{F^{\nr}}\B_{\st})^{\phi=p^r,N=0}
\verylomapr{\iota_{\dr}\otimes\iota}(\R\Gamma_{\dr}(X)\otimes_{\ovk}\B_{\dr})/F^r]\notag\\
 & \simeq [\R\Gamma_{\eet}(X,\Qp)\otimes_{\Qp}\B_{\crr}^{\phi=p^r}
\verylomapr{\iota_{\dr}\otimes\iota}\R\Gamma_{\eet}(X,\Qp)\otimes_{\Qp}\B_{\dr}/F^r]\notag\\
 & \stackrel{\sim}{\leftarrow} \R\Gamma_{\eet}(X,\Qp(r))\notag.
\end{align}
The last quasi-isomorphism follows from the fundamental exact sequence
$$
0\to \Qp(r)\to \B_{\crr}^{\phi=p^r}\to\B_{\dr}/F^r\to 0
$$
In a  stable range, the syntomic period map is a quasi-isomorphism.
   \begin{proposition}(\cite[Prop. 4.6]{NN})
   \label{NN}
   The syntomic period morphism induces a quasi-isomorphism
   $$\rho_{\synt}: \tau_{\leq r}\R\Gamma_{\synt}(X,r)\stackrel{\sim}{\to}\tau_{\leq r}\R\Gamma_{\eet}(X,\Qp(r)).
      $$
   \end{proposition}
   \subsubsection{Homotopy property}
Syntomic cohomology has  homotopy invariance property. 
\begin{proposition}
\label{homotopy}
 Let $X\in {\mathcal V}ar_{\ovk}$ and let $f: {\mathbb A}^1_X \to X$ be the natural projection from the affine line over $X$ to $X$. Then, 
 for all $r\geq 0$, the pullback map 
$$f^*:\,\R\Gamma_{\synt}(X,r)\lomapr{\sim}\R\Gamma_{\synt}({\mathbb A}^1_{X},r)$$
is a quasi-isomorphism.
\end{proposition}
\begin{proof}
It suffices to show that the pullback maps
\begin{align*}
f^*:(H^i_{\hk}(X)\otimes_{F^{\nr}}\B^+_{\crr})^{\phi=p^r} & \to (H^i_{\hk}({\mathbb A}^1_X)\otimes_{F^{\nr}}\B^+_{\crr})^{\phi=p^r},\\
f^*: (H^i_{\dr}(X)\otimes_{\ovk}\B^+_{\dr})/F^r & \to (H^i_{\dr}({\mathbb A}^1_X)\otimes_{\ovk}\B^+_{\dr})/F^r
\end{align*}
are isomorphisms. But this follows immediately from the fact that 
 we have a filtered isomorphism
$$f^*: H^i_{\dr}(X)\stackrel{\sim}{\to} H^i_{\dr}({\mathbb A}^1_X)
$$
and hence, via the Hyodo-Kato isomorphism, also a Frobenius equivariant isomorphism
$$f^*: H^i_{\hk}(X)\stackrel{\sim}{\to} H^i_{\hk}({\mathbb A}^1_X).
$$
\end{proof}
 \subsubsection{Projective space theorem}
   For $X\in {\mathcal V}ar_K$, we have the functorial syntomic Chern class map \cite[5.1]{NN}
\begin{align*}
c_1^{\synt}: \Pic(X) {\to }H^2_{\synt}(X,1).
\end{align*}
For $X\in {\mathcal V}ar_{\ovk}$, it yields the  syntomic Chern class map 
\begin{align*}
c_1^{\synt}: \Pic(X) {\to }H^2_{\synt}(X,1).
\end{align*}

 We  have the following projective space theorem for syntomic cohomology.
\begin{proposition}
\label{projective}
Let $\se$ be a locally free sheaf of rank $d+1$, $d\geq 0$, on a scheme $X\in {\mathcal V}ar_{\ovk}$. Consider the associated projective bundle $\pi:{\mathbb P}(\se)\to X$.  Then we have the following isomorphism 
\begin{align*}
\bigoplus_{i=0}^d{c}_1^{\synt}(\so(1))^i\cup\pi^*: \quad
\bigoplus_{i=0}^d H^{a-2i}(X,r-i)\stackrel{\sim}{\to} H^a(X,r), \quad 0\leq d \leq r.
\end{align*}
Here, the class ${c}_1^{\synt}(\so(1))\in H^2_{\synt}({\mathbb P}(\se), 1)$ refers to the class of the tautological bundle on ${\mathbb P}(\se)$.
\end{proposition}
\begin{proof}Just as in the proof of Proposition 5.2  from \cite{NN}, 
 the above projective space theorem can be reduced to the projective space theorems for the Hyodo-Kato and the Hodge cohomologies. We refer to loc. cit. for details and notation. 

 To prove our proposition it suffices to show that for any ss-pair $(U,\overline{U})$ over $K$ and the projective space $\pi: {\mathbb P}^d_{\overline{U}}\to\overline{U} $ of dimension $d$ over $\overline{U}$ we have a  projective space theorem for syntomic cohomology ($a\geq 0$)
\begin{align*}
\bigoplus_{i=0}^d{c}_1^{\synt}(\so(1))^i\cup\pi^*: \quad
\bigoplus_{i=0}^d H^{a-2i}_{\synt}((U,\overline{U})_{\ovk},r-i)\stackrel{\sim}{\to} H^a_{\synt}(({\mathbb P}^d_{U},{\mathbb P}^d_{\overline{U}})_{\ovk},r), \quad 0\leq d \leq r,
\end{align*}
where the class ${c}_1^{\synt}(\so(1))\in H^2_{\synt}(({\mathbb P}^d_{U},{\mathbb P}^d_{\overline{U}}), 1)$ refers to the class of the tautological bundle on ${\mathbb P}^d_{\overline{U}}$.

  By the distinguished triangle 
 \begin{equation*}
{\mathrm R}\Gamma_{\synt}((U,\overline{U})_{\ovk},r) \to
\R\Gamma_{\crr}((U,\overline{U})_{\ovk})^{\phi=p^r}\stackrel{}{\to} (\R\Gamma_{\dr}((U,\overline{U})_{\ovk})\otimes_{}\B^+_{\dr})/F^r
\end{equation*} and its compatibility with the action of $c_1^{\synt}$, it suffices to prove the following two isomorphisms for the  absolute log-crystalline complexes and for the filtered log de Rham complexes ($0\leq d \leq r$)
\begin{align}
\label{added}
\bigoplus_{i=0}^d{c}_1^{\crr}(\so(1))^i\cup\pi^*:  &\quad
\bigoplus_{i=0}^d H^{a-2i}_{\crr}((U,\overline{U})_{\ovk}) \stackrel{\sim}{\to} H^a_{\crr}(({\mathbb P}^d_{U},{\mathbb P}^d_{\overline{U}})_{\ovk}), \\
\bigoplus_{i=0}^d{c}_1^{\dr}(\so(1))^i\cup\pi^*: & \quad
\bigoplus_{i=0}^d (H^{a-2i}_{\dr}(U_{\ovk})\otimes_{}\B^+_{\dr})/F^{r-i} \stackrel{\sim}{\to} (H^a_{\dr}({\mathbb P}^d_{U,\ovk})\otimes_{}\B^+_{\dr})/F^r\notag.
\end{align}
For the crystalline cohomology in (\ref{added}), we can pass to the Hyodo-Kato cohomology. There the projective space theorem 
\begin{equation*}
\bigoplus_{i=0}^d{c}_1^{\hk}(\so(1))^i\cup\pi^*: \quad
\bigoplus_{i=0}^d H^{a-2i}_{\hk}((U,\overline{U})_{\ovk})\otimes_{F^{\nr}}\B^+_{\crr} \stackrel{\sim}{\to} H^a_{\hk}(({\mathbb P}^d_{U},{\mathbb P}^d_{\overline{U}})_{\ovk})\otimes_{F^{\nr}}\B^+_{\crr}\end{equation*}
 follows immediately, via the Hyodo-Kato isomorphism, from the projective space theorem for the de Rham cohomology.

 For the de Rham cohomology in (\ref{added}), passing to the grading we obtain
\begin{align}
\label{hkpst}
\bigoplus_{i=0}^d{c}_1^{\dr}(\so(1))^i\cup\pi^*: & \quad\bigoplus_{i=0}^d \gr^{r-i}(H^{a-2i}_{\dr}(U_{\ovk})\otimes_{\ovk}\B^+_{\dr}) \stackrel{\sim}{\to} \gr^{r}(H^a_{\dr}({\mathbb P}^d_{U,\ovk})\otimes_{\ovk}\B^+_{\dr}).
\end{align}
Since, for a variety $Y$ over $\ovk$, 
$$
\gr^{r}(H^a_{\dr}(Y)\otimes_{\ovk}\B^+_{\dr}) = \bigoplus_{i=0}^r
H^{a-r+i}_{\dr}(Y,\Omega^{r-i}_Y)\otimes_{\ovk}C,
$$
the isomorphism (\ref{hkpst})
follows from the projective space theorem for Hodge cohomology. We are done.
 \end{proof}
  \subsubsection{Bloch-Ogus Theory}

 The above implies that syntomic cohomology is representable by a motivic ring spectrum $\sss$: the argument is the same as in Appendix B of \cite{NN}. We list the following consequences.
  \begin{proposition}
\label{Bloch-Ogus}
\begin{enumerate}
\item Syntomic cohomology is covariant with respect
 to projective morphisms of smooth varieties. More precisely, to  a projective morphism of smooth $K$-varieties $f: Y\to X$ one can associate a Gysin morphism in syntomic cohomology
  $$f_*: H^i_{\synt}(Y,r)\to H^{i-2d}_{\synt}(X,r-d), $$
  where $d$ is the dimension of $f$. 
\item We have the syntomic regulator 
\begin{align*}
r_{\synt}:H^{r,i}_M(X) \to H^i_{\synt}(X,r),
\end{align*}
where $H^{r,i}_M(X)$ demotes the motivic cohomology. 
It is compatible with product, pullbacks, and pushforwards; via the period map it is compatible with the \'etale regulator.
\item The syntomic cohomology has a natural extension to 
 $h$-motives:
$$
DM_h(K,\Qp)^{op} \rightarrow D(\Qp), \quad M \mapsto \Hom_{DM_h(K,\Qp)}(M,\sss)
$$
and the syntomic regulator $r_{\synt}$ can be extended to motives.
\item There exists a canonical syntomic Borel-Moore homology
 $H^{\synt}_*(-,*)$ such that the pair of functor 
 $(H_{\synt}^*(-,*),H^{\synt}_*(-,*))$ defines a Bloch-Ogus theory.
\item To the ring spectrum $\sss$ there is associated a cohomology with
 compact support satisfying the usual properties.
\end{enumerate}
  \end{proposition}

\subsection{Fundamental (long) exact sequence}In this section, we will discuss certain exact sequences that involve syntomic cohomology. 
 Recall that we have   
 \begin{align}
 \label{mapfib}
 \R\Gamma_{\synt}(X,r)= & [\R\Gamma_{\crr}(X)^{\phi=p^r}\lomapr{\gamma}\R\Gamma_{\dr}(X)\otimes_{\ovk}\B^+_{\dr})/F^r]\\
  & [(\R\Gamma_{\hk}(X)\otimes_{F} \B^+_{\st})^{\phi=p^r,N=0}\lomapr{\iota_{\dr}\otimes\iota}(\R\Gamma_{\dr}(X)\otimes_{\ovk}\B^+_{\dr})/F^r]\notag
 \end{align}
 This yields a long exact sequence of cohomology that simplifies quite a bit. 
Indeed, we have 
 $H^j((\R\Gamma_{\hk}(X)\otimes_{F} \B^+_{\st})^{\phi=p^r,N=0})=(H^j_{\hk}(X)\otimes_{F}\B^+_{\st})^{\phi=p^r,N=0}$
  \cite[Corollary 3.25]{NN}. It follows  that $H^j\R\Gamma_{\crr}(X)^{\phi=p^r}=H^j_{\crr}(X)^{\phi=p^r}\simeq (\R\Gamma_{\hk}(X)\otimes_{F^{\nr}}\B_{\crr}^+)^{\phi=p^r}$. By the degeneration of the Hodge-de Rham spectral sequence, 
  we also have  $ H^j((\R\Gamma_{\dr}(X)\otimes_{\ovk}\B^+_{\dr})/F^r)=(H^j_{\dr}(X)\otimes_{\ovk}\B^+_{\dr})/F^r$.  Hence, from the mapping fiber (\ref{mapfib}),  we get the following fundamental long exact sequence
  $$\to H^{i-1}_{\crr}(X)^{\phi=p^r}\stackrel{\gamma_{i-1}}{\to} (H^{i-1}_{\dr}(X)\otimes_{\ovk}\B^+_{\dr}) /F^r\to H^i_{\synt}(X,r) \to H^i_{\crr}(X)^{\phi=p^r}\stackrel{\gamma_i}{\to}  (H^i_{\dr}(X)\otimes_{\ovk}\B^+_{\dr}) /F^r\to 
$$
that we will write alternatively as 
\begin{align}
\label{seq0}
\to (H^{i-1}_{\hk}(X)\otimes_{F}\B^+_{\st})^{\phi=p^r,N=0} & \stackrel{\gamma_{i-1}}{\to} (H^{i-1}_{\dr}(X)\otimes_{\ovk}\B^+_{\dr}) /F^r\to H^i_{\synt}(X,r) \\
  & \to (H^i_{\hk}(X)\otimes_{F}\B^+_{\st})^{\phi=p^r,N=0} \stackrel{\gamma_i}{\to}  (H^i_{\dr}(X)\otimes_{\ovk}\B^+_{\dr}) /F^r\to \notag
\end{align}
It yields  the exact sequence
\begin{equation}
\label{seq1}
0\to \coker \gamma_{i-1}\to H^i_{\synt}(X,r) \to  \ker \gamma_i\to 0.
\end{equation}
\begin{example}
Let $X=\Spec(\ovk)$. We get the exact sequence
$$
{\to} (H^{i-1}_{\dr}(\ovk)\otimes_{\ovk}\B^+_{\dr}) /F^r\to H^i_{\synt}(\ovk,r) \to H^i_{\crr}(\ovk)^{\phi=p^r}{\to}  (H^i_{\dr}(\ovk)\otimes_{\ovk}\B^+_{\dr}) /F^r\to 
$$
Since $H^i_{\crr}(\ovk)=H^i_{\crr}(F)\otimes_{F}\B^+_{\crr}$, we get that $H^0_{\crr}(\ovk)=\B^+_{\crr}$ and $H^i_{\crr}(\ovk)=0$, for $i>0$. Also, clearly, $H^0_{\dr}(\ovk)=\ovk$ and $H^i_{\dr}(\ovk)=0$, for $i >0$. Hence the above sequence becomes the fundamental exact sequence
$$
0\to \Qp(r)\to (\B^{+}_{\crr})^{\phi=p^r}\to \B^+_{\dr}/F^r\to 0
$$
It implies that $H^0_{\synt}(\ovk,r)\simeq \Qp(r)$ and $H^i_{\synt}(\ovk,r)=0$, for $i >0$.
\end{example}
\subsubsection{Relation to extensions of $\phi$-modules over $\ovk$}It turns out that the kernel and cokernel appearing in the exact sequence (\ref{seq1}) are extension groups in the category of filtered $\phi$-modules over $\ovk$. To see this, 
   set $D^i(r):= H^i_{\hk}(X)^{\tau}_{\B^+}$ with Frobenius $\phi_r=\phi/p^r$ and set $\Lambda^i(r):= F^r (H^i_{\dr}(X)\otimes_{\ovk}\B^+_{\dr})$. Let $H^i_{\G}(X,r):=(D^i(r),\Lambda^i(r))$. 
   Since $H^i_{\hk}(X)\otimes_{\B^+}\B^+_{\dr}\simeq H^i_{\dr}(X)\otimes_{\ovk}\B^+_{\dr}$ (via the Hyodo-Kato isomorphism), we have 
   $H^i_{\G}(X,r)\in \G^{+} $. 

\begin{lemma}
\label{long}
We have the following exact sequences
\begin{align*}
 & 0\to H^1_{+}(\ovk,H^{i-1}_{\G}(X,r))\to H^i_{\synt}(X,r) \to  H^0_{+}(\ovk,H^i_{\G}(X,r))\to 0.\\
& 0\to \kker(H^1(X_{\ff},\se(H^{i-1}_{\G}(X,r)))\to H^1(X_{\ff},\se(D^{i-1}(r))))
  \to H^i_{\synt}(X,r)\\
   & \quad \quad \to H^0(X_{\ff},\se(H^{i}_{\G}(X,r)))\to 0
\end{align*}

  Moreover, for $i\leq r+1$ or $r\geq d$, there are natural isomorphisms
$$
H^i_{\synt}(X,r) \stackrel{\sim}{\to } H^0_{+}(\ovk,H^{i}_{\G}(X,r))\simeq  H^0(X_{\ff},\se(H^{i}_{\G}(X,r))).
$$\end{lemma}
\begin{proof}
Since  $\ker \gamma_i=D^i(r)^{\phi_r=1}\cap\Lambda^i(r)$ and $\coker \gamma_{i}=(D^i(r)\otimes_{\B^+}\B^+_{\dr})/(\Lambda^i(r)+D^i(r)^{\phi_r=1})$, the first exact sequence follows from  (\ref{seq1}) and (\ref{form}).
The second exact sequence follows from that and 
  from  the exact sequence of sheaves on $X_{\ff}$ 
$$
0\to \se(H^{i}_{\G}(X,r))\to \se(D^i(r))\to i_{\infty *}(D^i(r)\otimes_{\B^+} \B^+_{\dr}/\Lambda^i(r))\to 0.
$$

 The last statement of the lemma follows from the first exact sequence and the fact that, for $i\leq r$ or $r\geq d$,   the filtered $\phi$-module $H^{i}_{\G}(X,r)$ is admissible. To see the last claim, note that  the variety $X$ comes from a variety $X_L$  defined over some finite extension $L$ of $K$. Hence, by comparison theorems, the pair $(H^i_{\hk}(X_L),H^i_{\dr}(X_L))$ forms  an admissible $(\phi,N,G_L)$-module.
Since $F^{r+1}H^i_{\dr}(X)=0$, we can now simply quote  \cite[Cor. 2.10]{NN}. \end{proof}
In the next section we will give a more conceptual reason for the existence  of the sequences from  the above lemma.
We note, that the above lemma implies that the fundamental (long) exact sequence (\ref{seq0})
 splits in a  stable range: for $i\leq r$, we have the fundamental short exact sequences
 $$
 0\to H^i_{\synt}(X,r)\to (H^i_{\hk}(X)\otimes_{F^{\nr}}\B^+_{\st})^{\phi=p^r,N=0}\verylomapr{\iota_{\dr}\otimes\iota}(H^i_{\dr}(X)\otimes_{\ovk}\B^+_{\dr})/F^r\to 0
 $$
\section{The $p$-adic absolute Hodge cohomology}We will show in this section that the geometric syntomic cohomology is a $p$-adic absolute cohomology,  that is, that 
to every variety over $\ovk$ one  can associate a canonical complex of effective filtered $\phi$-modules over $\ovk$ and syntomic cohomology is $\R\Hom$ from the trivial module to that complex. We will describe two methods of constructing such complexes. 
\subsection{Via geometric $p$-adic Hodge complexes}
  \subsubsection{The category of geometric $p$-adic Hodge complexes}
 Let $\Mod^{\prime}_{\B_{\dr}^+}$, $F\Mod^{\prime}_{\B_{\dr}^+}$ be the categories  $\Mod_{\B_{\dr}^+}$, $F\Mod_{\B_{\dr}^+}$ with the finiteness conditions dropped. These are  exact  categories.  Consider
the  exact monoidal functors 
$$F_0: \Mod^{\prime}_{\B^+}(\phi)\to \Mod_{\B^+_{\dr}},\quad  D\mapsto D\otimes_{\B^+}\B^+_{\dr}; \quad 
F_{\dr}: F\Mod^{\prime}_{\B^+_{\dr}}\to \Mod_{\B^+_{\dr}}, \quad (\Lambda,M)\mapsto M.
$$ 

We define the dg category $\sd_{\ph}$ of  $p$-adic Hodge complexes as the homotopy limit 
$$\sd_{\ph}:=\holim(\sd^b(\Mod_{\B^+}(\phi))\stackrel{F_0}{\to} \sd^b(\Mod^{\prime}_{\B^+_{\dr}})\stackrel{F_{\dr}}{\leftarrow}
\sd^b(F\Mod^{\prime}_{\B^+_{\dr}}))
$$
We denote by  $D_{\ph}$  the homotopy category of $\sd_{\ph}$.
An object of $\sd_{\ph}$ consists of objects $M_{0}\in \sd^b(\Mod_{\B^+}(\phi))$,
$M_{K}\in \sd^b(F\Mod^{\prime}_{\B^+_{\dr}})$,  and a quasi-isomorphism
$$
F_0(M_{0})\stackrel{a_M}{\to}  F_{\dr}(M_{K})$$
 in $\sd(\Mod^{\prime}_{\B^+_{\dr}})$. We will denote the object above by 
$
M=(M_0, M_{K}, a_M).
$
 The morphisms are given by the complex $ \Hom_{\sd_{\ph}}((M_0,M_{K},a_M),(N_0,N_{K},a_N))$:
\begin{equation}
\begin{split}
 \quad \quad  &\Hom^i_{\sd^g_{\ph}}((M_0,M_{K},a_M),(N_0,N_{K},a_N)) \\
  & \qquad\quad=\Hom^i_{\sd^b(\Mod_{\B^+}(\phi))}(M_0,N_0)\oplus \Hom^i_{\sd^b(F\Mod^{\prime}_{\B^+_{\dr}})}(M_{K},N_{K})\oplus \Hom^{i-1}_{\sd^b(\Mod^{\prime}_{\B^+_{\dr}})}(F_0(M_0),F_{\dr}(N_{K}))
\end{split}
\end{equation}
A (closed) morphism 
$(a,b,c)\in \Hom_{\sd_{\ph}}((M_0,M_{K},a_M),(N_0,N_{K},a_N))$ is a quasi-isomorphism if and only so are the morphisms $a$ and $b$.

By definition, we get a commutative square of  dg categories over $\Qp$:
\begin{equation} \label{eq:specialization_ph}
\xymatrix@=10pt{
\sd_{\ph}\ar^-{T_{\dr}}[r]\ar_{T_{0}}[d] & \sd^b(F\Mod^{\prime}_{\B^+_{\dr}})\ar^{F_{\dr}}[d] \\
\sd^b(\Mod_{\B^+}(\phi))\ar^-{F_0}[r] & \sd^b(\Mod^{\prime}_{\B^+_{\dr}}).
}
\end{equation}
 As pointed out above, a morphism $f$ of  $p$-adic Hodge complexes is a quasi-isomorphism
 if and only if $T_{\dr}(f)$ and $T_{0}(f)$ are quasi-isomorphisms.

  For  $M\in C^b(\G^{+} )$, we can define $\theta(M)\in \sd_{\ph}$ to be the object
$$
\theta(M):=
(D, (\Lambda,F_{0}(D)),\id:F_0(D)\simeq F_0(D)).
$$
Since $\theta$ preserves quasi-isomorphisms, it induces a compatible functor
$$
\theta: \quad\sd^b(\G^{+} )\to \sd_{\ph}.
$$

\begin{definition}
We will say that a $p$-adic Hodge complex $M=(M_0,M_K,a_M)$
 is \emph{geometric} if the complex $M_0$ is in the image of the canonical functor $C^b(\Mod_{F^{\nr}}(\phi))\to C^b(\Mod_{\B^+}(\phi))$ and the complex $M_K$ is strict with torsion-free finite rank cohomology groups (defined in the category of pairs of $\B^+_{\dr}$-modules). We note that we have the exact sequence (in the exact category $F\Mod^{\prime}_{\B^+_{\dr}}$)
\begin{equation}
 \label{strict}0\to \im d^{i-1} {\to} \ker d^i\stackrel{p}{\to} H^i\to 0.
\end{equation}
 Denote
  by $\sd_{\ph}^{g}$ the full dg subcategory of $\sd_{\ph}$ of geometric $p$-adic Hodge complexes.
\end{definition}
We note that for a geometric complex $M=(M_0,M_K,a_M)$ it makes sense to take its cohomology groups: 
$$H^i(M):=(H^i(M_0), H^i(M_K),a_M:F_0H^i(M_0)\simeq F_{\dr}H^i(M_K))\in \G^{+} , \quad i\geq 0. $$
\begin{proposition}
\label{thm1}
The functor $\theta$ induces an equivalence of dg categories
$$
\theta: \quad\sd^g(\G ^{+}) \stackrel{\sim}{\rightarrow} \sd^{g}_{\ph},
$$
where $\sd^g(\G ^{+})$ denotes the full dg subcategory of $\sd^b(\G ^{+})$ of geometric $p$-adic Hodge complexes.
\end{proposition}
\begin{proof}
First, we will show that $\theta$ is fully faithful. That is, that, given two complexes $M$, $M'$ of $C^g(\G^{+} )$,
 the functor $\theta$ induces a quasi-isomorphism:
$$
\theta:\quad \Hom_{\sd^b(\G^{+} )}(M,M')
 \rightarrow \Hom_{\sd^g_{\ph}}(\theta(M),\theta(M'))
$$ By Proposition \ref{comp1}, since $F_0(M_0)=F_{\dr}(M_K), F_0(M^{\prime}_0)=F_{\dr}(M^{\prime}_K)$, 
we have the following sequence of quasi-isomorphisms 
 \begin{align*}
     \Hom_{\sd^g_{\ph}}(\theta(M),\theta(M')) & =\Hom_{\sd^g_{\ph}}((M_0,M_{K},\id_M)  ,(M^{\prime}_0,M^{\prime}_{K},\id_{M^{\prime}}))\\
      & \simeq (\Hom_{\sd^b(\Mod_{\B^+}(\phi))}(M_0,M^{\prime}_0)\stackrel{F_0}{\to}  \Hom_{\sd^b(\Mod_{\B^+_{\dr}})}(F_0(M),F_{\dr}(M^{\prime})) \stackrel{F_{\dr}}{\leftarrow}  \Hom_{\sd^b(F\Mod_{\B^+_{\dr}})}(M_{K},M^{\prime}_{K}))\\
 & \simeq (\Hom_{\B^+,\phi}(M_0,M^{\prime}_0)\stackrel{F_0}{\to}  \Hom_{{\B^+_{\dr}}}(F_0(M),F_{\dr}(M^{\prime})) \stackrel{F_{\dr}}{\leftarrow}  \Hom_{F\Mod_{\B^+_{\dr}}}(M_{K},M^{\prime}_{K}))\\
 & \simeq \Hom_{\sd^b(\G^{+} )}(M,M^{\prime}).
 \end{align*}
 
 Now, it remains to show that $\theta$ is essentially surjective.  Let $(M_0,M_K,a_M)$ be a geometric $p$-adic Hodge complex. Then, by assumption,
 \begin{align}
 \label{correction}
 (M_0,M_K,a_M) \simeq 
 ( \oplus_{i\geq 0}H^i(M_0)[-i],& \oplus_{i\geq 0}H^i(M_K)[-i], \\
  & a^{\prime}_M:\oplus_{i\geq 0}F_0H^i(M_0)[-i]\simeq  \oplus_{i\geq 0}F_{\dr}H^i(M_K)[-i] ).\notag
 \end{align}
 Note that $a^{\prime}_M$ is actually an isomorphism of complexes. Essential surjectivity of $\theta$ follows.
 \end{proof}
 \begin{remark}
 We note that the proof of Proposition \ref{thm1} shows that every geometric $p$-adic Hodge complex is quasi-isomorphic to a complex of its cohomology filtered $\phi$-modules  (with trivial differentials). 
 \end{remark}

\subsubsection{  $p$-adic absolute Hodge cohomology via geometric $p$-adic Hodge complexes}
Let $X$ be a variety over $\ovk$. 
 For $r\geq 0$, consider the following complex in $\sd^{g}_{\ph}$
$$
\R\Gamma^g_{\ph}(X,r):=( \R\Gamma_{\hk}(X,r)_{\B^+}^{\tau}, (\R\Gamma_{\dr}(X)\otimes_{\ovk}\B^+_{\dr},F^r), \iota_{\dr}:\R\Gamma_{\hk}(X)_{\B^+}^{\tau}\otimes_{\B^+}\B^+_{\dr}\stackrel{\sim}{\to}\R\Gamma_{\dr}(X)\otimes_{\ovk}\B^+_{\dr}),
$$
 where the $r$-twist in $\R\Gamma_{\hk}(X,r)_{\B^+}^{\tau} $ refers to Frobenius divided by $p^r$. As explained at the beginning of Section \ref{kwak?}, the $p$-adic Hodge complex $\R\Gamma^g_{\ph}(X,r)$ is geometric. 
 We will call it
{\em the geometric $p$-adic Hodge cohomology of $X$}.
Set 
\begin{align*}
\R\Gamma_{\G}(X,r) & :=\theta^{-1}\R\Gamma^g_{\ph}(X,r)\in\sd^g(\G^{+} ).
\end{align*}

The {\em  $p$-adic absolute Hodge cohomology of $X$}  is  defined as
\begin{align}
\label{def1}
\R\Gamma_{\sh}(X,r):=\Hom_{\sd^g_{\ph}}(\un,\R\Gamma^g_{\ph}(X,r)).
\end{align}
By Proposition \ref{thm1}, we have 
\begin{align*}
\R\Gamma_{\sh}(X,r) \simeq \Hom_{\sd^b(\G^{+} )}(\un,\R\Gamma_{\G}(X,r)).
\end{align*}

\begin{theorem}
\label{compsynph}
There exists a natural quasi-isomorphism (in the classical derived category)
$$\R\Gamma_{\synt}(X,r)\stackrel{\sim}{\to} \R\Gamma_{\sh}(X,r), \quad r\geq 0.
$$
\end{theorem}
\begin{proof}We can write
   \begin{align*}
\R\Gamma_{\synt}(X,r)
  \stackrel{\sim}{\to}
  [\xymatrix{\R\Gamma_{\hk}(X)_{\B^+}^{\tau,\phi=p^r}\oplus F^r(\R\Gamma_{\dr}(X)\otimes_{\ovk}\B^+_{\dr})  \ar[rr]^-{\iota_{\dr}\otimes\iota-\can} & &  \R\Gamma_{\dr}(X)\otimes_{\ovk}\B^+_{\dr}}]
 \end{align*}
 By Proposition \ref{thm1}, we have 
 \begin{align*}
\R\Gamma_{\synt}(X,r)\simeq \Hom_{\sd^g_{\ph}}(\un,\R\Gamma^g_{\ph}(X,r))\simeq
 \Hom_{\sd^b(\G^{+} )}(\un,\R\Gamma_{\G}(X,r))\simeq  \R\Gamma_{\sh}(X,r),
\end{align*}
as wanted.
\end{proof}
\subsection{Via Beilinson's Basic Lemma}
   Using Beilinson's  Basic Lemma as in \cite{DN} we can associate to every $X\in {\mathcal V}ar_{\ovk}$ (more generally, to any finite diagram of such $X$) the following functorial data
   \begin{enumerate}
   \item $\R\Gamma^B_{\hk}(X)$, a complex of finite $(\phi,N)$-modules over $F^{\nr}$ representing $\R\Gamma_{\hk}(X)$.
   \item $\R\Gamma^B_{\dr}(X)$, a complex of finite filtered $\ovk$-vector spaces representing $\R\Gamma_{\dr}(X)$.
   \item the Hyodo-Kato isomorphism of complexes
   $$
   \iota_{\dr}: \R\Gamma^B_{\hk}(X) \otimes_{F^{\nr}}\ovk  \stackrel{\sim}{\to}\R\Gamma^B_{\dr}(X) 
   $$
   representing the original Hyodo-Kato map $\iota_{\dr}$.
   \end{enumerate}
They yield a    
    functorial complex
$\R\Gamma^B_{\G}(X,r)$, $r\geq 0$,  of effective filtered $\phi$-modules    $(\R\Gamma^B_{\hk}(X)_{\B^+}^{\tau},(\R\Gamma^B_{\hk}(X)\otimes_{\ovk}\B^+_{\dr},F^r),\iota_{\dr})$.
By construction,  for every good pair $(X,Y,i)$, the associated complex
$$
\R\Gamma^B_{\G}(X,Y,r)\simeq H^i_{\G}(X,Y,r):=(H^i_{\hk}(X,Y,r)^{\tau}_{\B^+},(H^i_{\dr}(X,Y)\otimes_{\ovk}B^+_{\dr},F^r),\iota_{\dr}).
$$
We note that the complex $\R\Gamma^B_{\G}(X,r)$ is geometric. Indeed,  Hodge Theory yields that "geometric" morphisms between de Rham cohomology groups are strict for the Hodge-Deligne filtration. Hence the complex $\R\Gamma^B_{\dr}(X)$ has strict differentials.
This implies that this is also the case for the complex $\R\Gamma^B_{\dr}(X)\otimes_{\ovk}\B^+_{\dr}$. Moreover, the cohomology groups  $H^i\R\Gamma^B_{\G}(X,r)$ are effective filtered $\phi$-modules.  Hence $\R\Gamma^B_{\G}(X,r)\in \sd^g(G^{+} )$.

 We set
$$
\R\Gamma^B_{\sh}(X,r)=\R\Gamma^B_{\synt}(X,r):=\R\Gamma_+(\ovk,\R\Gamma^B_{\G}(X,r))
$$
The two syntomic complexes described above are naturally quasi-isomorphic. 
\begin{theorem}
\begin{enumerate}
\item There is a canonical quasi-isomorphism in $\sd^b(\G^{+} )$
$$\R\Gamma_{\G}(X,r)\stackrel{\sim}{\to}\R\Gamma^B_{\G}(X,r), \quad \quad r\in\N.$$
\item There is a  canonical quasi-isomorphism 
$$
\R\Gamma_{\sh}(X,r)\simeq \R\Gamma_{\sh}^B(X,r),\quad r\in\N.
$$
\end{enumerate}
\end{theorem}
\begin{proof}The second statement follows immediately from the first one and Theorem \ref{compsynph}. 
  To prove the first statement, we follow the proof of Corollary 3.7 from \cite{DN} and just briefly describe the argument here.
  Consider the complex $\R\Gamma^B_{\ph}(X,r)$ in $\sd^g_{\ph}$  defined using Beilinson's  Basic Lemma  starting with  $ \R\Gamma_{\ph}(X,r)$. We note that,  for a good pair $(X,Y,j)$, we have $$\R\Gamma^g_{\ph}((X,Y,j),r)\simeq 
(H^j_{\hk}(X,Y,r)_{\B^+ }^{\tau},(H^j_{\dr}(X,Y)\otimes_{\ovk}\B^+_{\dr}, F^r),
 \iota_{\dr}: H^j_{\hk}(X,Y)^{\tau}_{\B^+ }\otimes_{\B^+ }\B^+_{\dr}\stackrel{\sim}{\to}H^j_{\dr}(X,Y)\otimes_{\ovk}\B^+_{\dr}).$$ Hence $\R\Gamma^B_{\ph}(X,r)$ is isomorphic to $\theta\R\Gamma^B_{\G}(X,r)$. Moreover, we get  a functorial quasi-isomorphism in $\sd^b(\G^{+} )$:  $$\kappa_X: \quad \R\Gamma^B_{\ph}(X,r)\simeq \theta\R\Gamma_{\G}(X,r).
$$ 
It follows that $\R\Gamma^B_{\G}(X,r)\simeq \R\Gamma_{\G}(X,r)$, as wanted. 
\end{proof}
\section{Syntomic  cohomology and  Banach-Colmez spaces}
We will show in this section that syntomic complex (of an algebraic variety) can be realized as a complex of Banach-Colmez spaces. Similarly for the fundamental exact sequence and the syntomic period map: both of them can be canonically lifted to the category of Banach-Colmez spaces. We will also discuss an analog for formal schemes.
\subsection{Syntomic  complex as a complex of Banach-Colmez spaces.}
\subsubsection{Algebraic varieties}
Let $X$ be a variety over $\ovk$. 
\begin{theorem}
\label{BC1}
\begin{enumerate}
\item There exists a complex ${\mathbb R}\Gamma_{\synt}^B(X,r)\in \sd^b(\BC)$ such that ${\mathbb R}\Gamma_{\synt}^B(X,r)(C)=\R\Gamma_{\synt}^B(X,r)$
and the distinguished triangle 
\begin{equation*}
  \R\Gamma_{\synt}(X,r)\to \R\Gamma_{\hk}(X)^{\tau,\phi=p^r}_{\B^+ }\lomapr{\iota_{\dr}} (\R\Gamma_{\dr}(X)\otimes_{\ovk}\B^+_{\dr})/F^r
   \end{equation*}
   can be lifted canonically to a distinguished triangle of  Banach-Colmez  spaces.
   \item The syntomic period map $\rho_{\synt}$ can be lifted canonically to a map of Banach-Colmez  spaces, i.e., we have a map
  \begin{align}
  \label{map1}
  \rho_{\synt}: {\mathbb R}\Gamma^B_{\synt}(X,r)  \to {\mathbb R}\Gamma^B_{\eet}(X,\Qp(r))
  \end{align}
 such that the induced map on $C$-points is the classical syntomic period map.  Here the complex ${\mathbb R}\Gamma^B_{\eet}(X,\Qp(r))$ is a a complex of finite-rank $\Q_p$-vector spaces defined using Beilinson's Basic Lemma. 
 \item We have a canonical  spectral sequence
$$
E_2^{i,j}:=H^i_{\!+}(\ovk, H^j_{\G}(X,r))\Rightarrow H^{i+j}_{\synt}(X,r)
$$
that degenerates at $E_2$. It can be canonically lifted to the category of 
Banach-Colmez spaces. 
 \end{enumerate}
\end{theorem}
\begin{proof}
 For the first claim, define
  \begin{equation}
  \label{sp}
  \R\Gamma_{\synt}^B(X,r)=[\R\Gamma^B_{\hk}(X)^{\tau,\phi=p^r}_{\B^+ }\lomapr{\iota_{\dr}} (\R\Gamma^B_{\dr}(X)\otimes_{\ovk}\B^+_{\dr})/F^r]
   \end{equation}
   By Section \ref{kwak2}, 
   \begin{align*}
   \R\Gamma^B_{\hk}(X)^{\tau,\phi=p^r}_{\B^+ } & ={\mathbb X}^r_{\st}( \R\Gamma^B_{\hk}(X))(C),\\
   (\R\Gamma^B_{\dr}(X)\otimes_{\ovk}\B^+_{\dr})/F^r &  ={\mathbb X}^r_{\dr}( \R\Gamma^B_{\dr}(X))(C),
    \end{align*}  
    and the map $\iota_{\dr}$ can be lifted canonically to a map between complexes of Banach-Colmez  spaces. We can now define the following complex of Banach-Colmez spaces
    \begin{align}
    \label{sp1}
    {\mathbb R}\Gamma_{\synt}^B(X,r) :=[{\mathbb X}^r_{\st}( \R\Gamma^B_{\hk}(X))\lomapr{\iota_{\dr}}{\mathbb X}^r_{\dr}( \R\Gamma^B_{\dr}(X))].   
    \end{align}
   We have $ {\mathbb R}\Gamma_{\synt}^B(X,r)(C)= \R\Gamma_{\synt}^B(X,r)$ and hence $H^*({\mathbb R}\Gamma_{\synt}^B(X,r))(C)=H^i\R\Gamma_{\synt}^B(X,r)$.

   To define the lift (\ref{map1}), it suffices to inspect the sequence of quasi-isomorphisms (\ref{quasi}) defining the syntomic period map and to 
 \begin{enumerate}
 \item replace the complexes $\R\Gamma_{\hk}(X)$, $\R\Gamma_{\dr}(X)$, and $\R\Gamma_{\eet}(X,\Qp)$ by their analogs $\R\Gamma^B_{\hk}(X)$, $\R\Gamma^B_{\dr}(X)$, and 
 $\R\Gamma^B_{\eet}(X,\Qp)$ defined using Beilinson's Basic Lemma;
 \item note that $(\R\Gamma^B_{\hk}(X)\otimes_{F^{\nr}}\B_{\st})^{\phi=p^r,N=0}$ and  $(\R\Gamma^B_{\dr}(X)\otimes_{\ovk}\B_{\dr})/F^r$ are $C$-points of complexes of  Ind-Banach-Colmez spaces; and similarly for $\R\Gamma^B_{\eet}(X,\Qp)\otimes_{\Qp}\B_{\st}^{\phi=p^r,N=0}$ and $\R\Gamma^B_{\eet}(X,\Qp)\otimes_{\Qp}\B_{\dr}/F^r$;
 \item note that, for a pst-pair $(D,V)$, the period isomorphisms
 $$(D\otimes_{F^{\nr}}\B_{\st})^{\phi=p^r,N=0}\simeq V\otimes_{\Q_p}\B_{\st}^{\phi=p^r,N=0},\quad (D_K\otimes_K\B_{\dr})/F^r\simeq V\otimes_{\Q_p}\B_{\dr}/F^r
 $$
 can be lifted canonically  to the category of Ind-Banach-Colmez spaces.
\end{enumerate}
  The spectral sequence in the third  claim is constructed from (\ref{sp}) and (\ref{sp1}) and uses the computations of extensions in effective filtered $\phi$-modules from (\ref{form}).
 
\end{proof}
  \subsubsection{ Syntomic Euler characteristic}    Let $X$ be a variety over $\ovk$.  By Section \ref{kwak2}, the syntomic Euler characteristic of $X$ is equal to
\begin{align*}
\chi({\mathbb R}\Gamma^B_{\synt}(X,r)) & =\sum_{i\geq 0}(-1)^i\dim {\mathbb X}^r_{\rm st}(D^i)-\sum_{i\geq 0}(-1)^i \dim {\mathbb X}^r_{\rm dR}(D^i)\\
 & =\sum_{i\geq 0}(-1)^i \sum_{r^{\prime}_i\leq r}(r-r^{\prime}_i,1)-\sum_{i\geq 0}(-1)^i (r\dim_{\ovk}D^i_{\ovk}-\sum_{j=1}^r\dim F^jD^i_{\ovk},0),\\
 \end{align*}
 where $D^i=(D^i,D^i_{\ovk})=(H^i_{\hk}(X),H^i_{\dr}(X))$. For $r$ large enough this stabilizes. That is, 
 if $F^{r+1}D^i_{\ovk}=0$, $i\geq 0$,  and if all $r^{\prime}_i$'s are $\leq r$, then
\begin{align*}
\chi({\mathbb R}\Gamma^B_{\synt}(X,r))
 & =\sum_{i\geq 0}(-1)^i (r\dim_{F^{\nr}}D^i-t_N(D^i),\dim_{F^{\nr}}D^i)-\sum_{i\geq 0}(-1)^i (r\dim_KD^i_{K}-t_H(D^i_{K}),0)\\
 &=\sum_{i\geq 0}(-1)^i(t_H(D^i_{K} )-t_N(D^i),  \dim_{F^{\nr}}D^i)=(\sum_{i\geq 0}(-1)^i(t_H(D^i_{K} )-t_N(D^i)),  \chi(\R\Gamma_{\dr}(X))).
 \end{align*}
 Since the filtered $(\phi,N,G_K)$-module $D^i=(D^i,D^i_{\ovk})$ is admissible, we have $t_N(D^i)=t_H(D^i_{K})$ and $\chi({\mathbb R}\Gamma^B_{\synt}(X,r))=(0,\chi(\R\Gamma_{\dr}(X)))$.
 \subsubsection{Formal schemes}Part of Theorem \ref{BC1} has an analog for formal schemes. 
 Let $\sx$ be a quasi-compact ($p$-adic) formal scheme over $\so_K$ with strict semistable reduction. As shown in \cite{CN}, for $r\geq 0$, one can associate to $\sx$ a functorial syntomic cohomology complex
 $\R\Gamma_{\synt}(\sx_{\ovk},r)$ defined be a formula analogous to (\ref{analog1}).  We see $\sx$ as a log-scheme and the Hyodo-Kato and the de Rham cohomologies used are logarithmic. 
 \begin{proposition}Assume that $\sx$ is proper.
 \begin{enumerate}
 \item 
 Then there exists a complex ${\mathbb R}\Gamma_{\synt}(\sx_{\ovk},r)\in \sd^b(\BC)$ such that ${\mathbb R}\Gamma_{\synt}(\sx_{\ovk},r)(C)=\R\Gamma_{\synt}(\sx_{\ovk},r)$
and the distinguished triangle 
\begin{equation*}
  \R\Gamma_{\synt}(\sx_{\ovk},r)\to \R\Gamma_{\hk}(\sx)^{\tau,\phi=p^r}_{\B^+ }\lomapr{\iota_{\dr}} (\R\Gamma_{\dr}(\sx_{\ovk})\otimes_{\ovk}\B^+_{\dr})/F^r
   \end{equation*}
   can be lifted canonically to a distinguished triangle of  Banach-Colmez  spaces.
  \item We have a canonical  spectral sequence
$$
E_2^{i,j}:=H^i_{\!+}(\ovk, H^j_{\G}(\sx_{\ovk},r))\Rightarrow H^{i+j}_{\synt}(\sx_{\ovk},r)
$$
that degenerates at $E_2$. It can be canonically lifted to the category of Banach-Colmez spaces. 
 \end{enumerate}
 \end{proposition}
 \begin{proof}
 Consider the following complex in $\sd^{g}_{\ph}$
$$
\R\Gamma^g_{\ph}(\sx_{\ovk},r):=( \R\Gamma_{\hk}(\sx,r)_{\B^+}^{\tau}, (\R\Gamma_{\dr}(\sx_{\ovk})\otimes_{\ovk}\B^+_{\dr},F^r), \iota_{\dr}:\R\Gamma_{\hk}(\sx)_{\B^+}^{\tau}\otimes_{\B^+}\B^+_{\dr}\stackrel{\sim}{\to}\R\Gamma_{\dr}(\sx_{\ovk})\otimes_{\ovk}\B^+_{\dr}),
$$
 where the $r$-twist in $\R\Gamma_{\hk}(\sx,r)_{\B^+}^{\tau} $ refers to Frobenius divided by $p^r$. The $p$-adic Hodge complex $\R\Gamma^g_{\ph}(\sx_{\ovk},r)$ is geometric: the explanation at the beginning of Section \ref{kwak?} goes through once we have the degeneration of the Hodge-de Rham spectral sequence and this was proved in \cite[Theorem 8.4]{Sch}.
Set 
\begin{align*}
\R\Gamma_{\G}(\sx_{\ovk},r) & :=\theta^{-1}\R\Gamma^g_{\ph}(\sx_{\ovk},r)\in\sd^g(\G^{+} ).
\end{align*}
We set $H^i_{\G}(\sx_{\ovk},r):=H^i\R\Gamma_{\G}(\sx_{\ovk},r)$.

  By Proposition \ref{thm1}, we have 
$$\Hom_{\sd^g_{\ph}}(\un,\R\Gamma^g_{\ph}(\sx_{\ovk},r))\simeq  \Hom_{\sd^b(\G^{+} )}(\un,\R\Gamma_{\G}(\sx_{\ovk},r)).
$$
Since    \begin{align*}
\R\Gamma_{\synt}(\sx_{\ovk},r)
  & \stackrel{\sim}{\to}
  [\xymatrix{\R\Gamma_{\hk}(\sx)_{\B^+}^{\tau,\phi=p^r}\oplus F^r(\R\Gamma_{\dr}(\sx_{\ovk})\otimes_{\ovk}\B^+_{\dr})  \ar[rr]^-{\iota_{\dr}\otimes\iota-\can} & &  \R\Gamma_{\dr}(\sx_{\tr})\otimes_{\ovk}\B^+_{\dr}}]\\
   & \simeq \Hom_{\sd^g_{\ph}}(\un,\R\Gamma^g_{\ph}(\sx_{\ovk},r)),
    \end{align*}
we get that 
 \begin{align*}
\R\Gamma_{\synt}(\sx_{\ovk},r)\simeq 
 \Hom_{\sd^b(\G^{+} )}(\un,\R\Gamma_{\G}(\sx_{\ovk},r)).
\end{align*}
The first claim of our proposition follows now from formula (\ref{correction}), Proposition \ref{comp1}, and  (\ref{bcspaces}). The second -- just as in the proof of Theorem \ref{BC1}.
 \end{proof}
\subsection{Identity component of syntomic cohomology}We will show that the image of the syntomic cohomology Banach-Colmez space in the \'etale cohomology is the $\Q_p$-vector space of its connected components. 
\begin{proposition}For a variety $X\in  {\mathcal V}ar_{\ovk}$ and $i\geq 0$, we have the natural isomorphisms
$$ \ker(\rho^i_{\synt})\simeq H^1_{+}(\ovk, H^{i-1}_{\G}(X,r)),\quad \im(\rho^i_{\synt})\simeq H^0_{+}(\ovk, H^{i}_{\G}(X,r)),
$$
where $\rho^i_{\synt}: H^i_{\synt}(X,r)\to H^i_{\eet}(X,\Q_p(r))$ is the syntomic period map.
\end{proposition}
\begin{proof}
   Consider the following commutative diagram
$$
\xymatrix{
\ar[r] & H^i_{\synt}(X,r)\ar[r]\ar[d]^{\rho^i_{\synt}} & H^i_{\crr}(X)^{\phi=p^r}\ar[r]^-{\gamma_i}\ar[d]^{\rho^i_{\crr}}  & (H^i_{\dr}(X)\otimes_{\ovk}\B_{\dr}^+)/F^r\ar[d]^{\rho^i_{\dr}}\ar[r] &\\
0\ar[r]  & H^i_{\eet}(X,\Qp(r))\ar[r] & H^i_{\eet}(X,\Qp)\otimes \B_{\crr}^{\phi=p^r}\ar[r]^{} & H^i_{\eet}(X,\Qp)\otimes_{\ovk}\B_{\dr}/F^r \ar[r] & 0
 }
 $$
 It follows that $\rho^i_{\synt}$ factors through $H^0_{+}(\ovk,H^i_{\G}(X,r))$ and that $H^0_{+}(\ovk, H^i_{\G}(X,r))\hookrightarrow H^i_{\eet}(X,\Qp(r))$. Hence $H^1_{+}(\ovk,H^{i-1}_{\G}(X,r))\simeq \ker(\rho^i_{\synt})$ and $H^0_{+}(\ovk, H^{i}_{\G}(X,r))\simeq \im(\rho^i_{\synt})$, as wanted.
 \end{proof}
 \begin{corollary}
 \label{long1}
 For $i\leq r$ or $r\geq d$, there are natural isomorphisms
$$
H^i_{\synt}(X,r) \stackrel{\sim}{\to } H^0_{+}(\ovk,H^{i}_{\G}(X,r))\lomapr{\rho^i_{\synt}} H^i_{\eet}(X,\Q_p(r)).
$$
 \end{corollary}
 \begin{proof}
 The first isomorphism follows from Lemma \ref{long}. The second, from the above proposition and the fact that, for $i\leq r$ or $r\geq d$, the pair $(H^i_{\hk}(X),H^i_{\dr}(X))$ comes from  an admissible filtered $(\phi,N,G_L)$-module (for a finite extension $L/K$) hence we can revoke \cite[Prop. 2.10]{NN}.
 \end{proof}
 The exact sequence of Banach-Colmez spaces
 $$
 0 \to H^1_{+}(\ovk, H^{i-1}_{\G}(X,r)) \to H^i{\mathbb R}\Gamma_{\synt}^B(X,r)\to H^0_{+}(\ovk, H^{i}_{\G}(X,r)) \to 0
 $$
 has the following interpretation. 
 The Banach-Colmez space $H^1_{+}(\ovk, H^{i-1}_{\G}(X,r))$ is connected (as a quotient of the connected Space ${\mathbb X}^r_{\dr}(H^{i-1}_{\dr}(X))$) and the Space $H^0_{+}(\ovk, H^{i}_{\G}(X,r))$ is a finite rank $\Q_p$-vector space. Hence $H^1_{+}(\ovk, H^{i-1}_{\G}(X,r))$ is the identity  component of $H^i{\mathbb R}\Gamma_{\synt}^B(X,r)$ and $H^0_{+}(\ovk, H^{i}_{\G}(X,r))=\pi_0(H^i{\mathbb R}\Gamma_{\synt}^B(X,r))$. The projection of Banach-Colmex spaces $H^i{\mathbb R}\Gamma_{\synt}^B(X,r)\to H^0_{+}(\ovk, H^{i}_{\G}(X,r))$ has a noncanonical section.
 \section{Computations}In this section we will present several computations of  geometric syntomic cohomology groups. They show that these groups yield interesting invariants of algebraic varieties. We do not, however,  understand what is going on. 
 \subsection{Ordinary varieties}We will show that geometric syntomic cohomology of ordinary varieties is a finite rank $\Q_p$-vector space. For varying twists $r$, it yields a Galois equivariant filtration of the \'etale cohomology. 
 
  Recall the following terminology from \cite{PR}. 
 A $p$-adic $G_K$-representation $V$  is called {\em ordinary} if it admits an equivariant increasing filtration $
  V_j, j\in \Z$, such that an open subgroup of the inertia group acts on $ V_j/V_{j-1}$ by $\chi^{-j}$, $\chi$ being the cyclotomic character.
  
 A filtered $(\phi,N,G_K)$-module is called {\em ordinary} if its Newton and Hodge polygons agree.  Such a module admits an increasing filtration by filtered $(\phi,N,G_K)$-submodules, 
 $D_j, j\in \Z$, such that $D_j/D_{j-1}=D_{[-j]}$. Here we set $D_{[-j]}:=(D_{[-j]},D_{[-j],\ovk})$, where $D_{[-j]}=D\cap (D\otimes_{F^{\nr}}W(\overline{k}))_{[-j]}$ for $(D\otimes_{F^{\nr}}W(\overline{k}))_{[-j]}$ --  a submodule of $D\otimes_{F^{\nr}}W(\overline{k})$ generated by $x$ such that $\phi(x)=p^{-j}x$. It is equipped with the trivial monodromy and  the filtration on $D_{[-j],K}:=D_{[-j],\ovk}^{G_K}$ is given by $F^{-j}D_K=D_K, F^{-j+1}D_K=0$. The categories of ordinary Galois representations and ordinary filtered $(\phi,N,G_K)$-modules are equivalent. We have $V_{\pst}(D_{[-j]})\simeq V_{-j}$. 
 
      A variety $X$ over $K$  is {\em ordinary} if so are  the pairs $(H^i_{\hk}(X),H^i_{\dr}(X))$, $i\geq 0$.   A variety $X$ over $\ovk$  is {\em ordinary} if it comes from an ordinary variety defined over a finite extension of $K$.
  \begin{example}
  \label{example-ordinary}For $r\geq 0$ and  $i\geq 0$, 
  we have $$
  H^i_{\synt}(X,r)=H^0_+(\ovk,H^i_{\G}(X,r))\stackrel{\sim}{\to}  H^i_{\eet}(X,\Q_p)_{r}(r)\hookrightarrow  H^i_{\eet}(X,\Q_p(r)).
  $$
  In particular,  it is a finite rank $\Q_p$-vector space. 
  \end{example}
  \begin{proof}
  For $k\geq 0$, let $D=F^{\nr}e, \phi(e)=p^ke, F^kD_{K}=D_{K}, F^{k+1}D_{K}=0$. Equip it with the trivial action of the monodromy and  $G_K$. The pair $(D,D_{\ovk})$ forms an admissible filtered $(\phi,N, G_K)$-module  whose associated Galois representation is $\Q_p(-k)$.

   If $r <k$, we have
   $$
   (D\otimes_{F^{\nr}} \B^+_{\crr})^{\phi=p^r}=   ( \B^+_{\crr})^{\phi=p^{r-k}}=0,\quad 
    (D\otimes_{F^{\nr}}\B^+_{\dr})/F^r=    \B^+_{\dr}/\B^+_{\dr}=0.
   $$
   Hence $H^0_+(\ovk,D)=H^1_+(\ovk,D)=0$. 
   
   If $r\geq k$, we have
   $$
      [(D\otimes_{F^{\nr}} \B^+_{\crr})^{\phi=p^r}\to  (D\otimes_{F^{\nr}}\B^+_{\dr})/F^r]=
      [(\B^+_{\crr})^{\phi=p^{r-k}}\to \B^+_{\dr}/F^{r-k}]\stackrel{\sim}{\leftarrow}\Q_p(r-k).
      $$  
      Hence $H^0_+(\ovk,D)=\Q_p(r-k)$ and $H^1_+(\ovk,D)=0$. 
      
  By devissage, it follows that, for any ordinary filtered $(\phi,N,G_K)$-module $D$, we have $H^1_+(\ovk,D)=0$ and $H^0_+(\ovk,D)\stackrel{\sim}{\to}V_{\st}(D_r)(r)=V_{r}(r)$, as wanted.
  \end{proof}
  We note that in the \'etale cohomology group $H^i_{\eet}(X,\Q_p)$ we see the cyclotomic characters $\chi^j, 0\geq j\geq -i$; hence in the group $H^i_{\eet}(X,\Q_p(r))$ - the characters $\chi^j,r\geq j\geq  -i+r$. 
  The above  example shows that syntomic cohomology $H^i_{\synt}(X,r)$, $r\geq 0$, picks   up the twists $\Q_p(j)$, for $r\geq  j\geq 0$, in the \'etale cohomology $H^i_{\eet}(X,\Q_p(r))$. If $r$ is large, $r\geq i$, we pick up the whole \'etale cohomology group.

\subsection{Curves}
 Let $X$ be a proper smooth curve over $\ovk$. We will show that the geometric syntomic cohomology of $X$ recovers  the " \'etale  $p$-rank of the special fiber of the N\'eron model of the Jacobian of the curve".  
 
  By Lemma \ref{long}, 
 $$
 H^i_{\synt}(X,r)=H^0_+(\ovk,H^i_{\G}(X,r)), \quad r\geq 1, i\leq r+1,
 $$
 is a finite rank vector space over $\Q_p$ that is a subspace of $H^i_{\eet}(X,\Q_p(r))$. Moreover,  
$$
 \rho_{\synt}^i: H^i_{\synt}(X,r)\stackrel{\sim}{\to}H^i_{\eet}(X,\Q_p(r)), \quad r\geq 1, i\leq r.
 $$
 For $r=0$,   we have 
 \begin{align*}
 H^0_{\synt}(X,0) & \stackrel{\sim}{\to}H^0_{\eet}(X,\Q_p(0))\simeq \Q_p.
 \end{align*}
 Since $F^0(H^i_{\dr}(X)\otimes\B^+_{\dr})= H^i_{\dr}(X)\otimes\B^+_{\dr}$, we have 
 \begin{align*} 
   H^i_{\synt}(X,0)\simeq (H^i_{\hk}(X)\otimes_{F^{\nr}}\B^+_{\crr})^{\phi=1}  \hookrightarrow H^i_{\eet}(X,\Q_p(0)).  
    \end{align*}
 Since $H^{2}_{\hk}(X)\simeq F^{\nr}$ with Frobenius $p\phi$, we have  $H^2_{\synt}(X,0)=0$ since
  $$
 H^2_{\synt}(X,0)\hookrightarrow (\B^+_{\crr})^{\phi=p^{-1}}=0.
 $$

   It remains to understand the inclusion $  H^1_{\synt}(X,0) \hookrightarrow H^1_{\eet}(X,\Q_p)$. Since the Frobenius slopes on $H^1_{\hk}(X)$ are $\geq 0$, we have 
   $$
   H^1_{\synt}(X,0)\simeq (H^1_{\hk}(X)\otimes_{F^{\nr}}\B^+_{\crr})^{\phi=1}=(H^1_{\hk}(X)_{[0]} \otimes_{F^{\nr}}\B^+_{\crr})^{\phi=1}= (H^1_{\hk}(X)_{[0]}\otimes_{F^{\nr}}W(\overline{k}))^{\phi=1}.
   $$
  This is a $\Q_p$-vector space, of the same rank as the $F^{\nr}$-rank of $H^1_{\hk}(X)_{[0]}$, that injects  via the period map  into $H^1_{\eet}(X,\Q_p)$.
   Since  $H^1_{\eet}(X,\Q_p)\simeq V_p(\Jac(X))^*$ (the $\Q_p$-dual), this rank is the " \'etale  $p$-rank of the special fiber of the N\'eron model of the Jacobian of the curve".
    \subsection{Product of elliptic curves}Varieties that are not ordinary give rise to geometric syntomic cohomology groups that have nontrivial $C$-dimension. We will start with some preliminary computations. 
   \subsubsection{The group $H^0_+(\cdot)$}\label{thegroup}
Let $D$ be an admissible filtered $(\phi,N,G_K)$-module with $F^0D_K=D_K$ (this implies that the slopes of Frobenius are $\geq 0$). Then the group $H^0_+((D\otimes_{F^{\nr}}\B^+_{\st})^{N=0,\phi=p^r},F^r(D_K\otimes_K\B^+_{\dr})))$ is the largest subrepresentation $V_r$ of $V_{\st}(D)(r)$ with Hodge weights in $[0,r]$.
It corresponds 
to the largest (weakly) admissible submodule $D_r$ of $D$ with $F^{r+1}D_{r,K}=0$   (note that this is equivalent to all Hodge weights being less or equal to $r$ and it implies that all Frobenius slopes have the same property). 
   
  \subsubsection{Trace map}\label{trace}
  Let $X$ be a variety over $\ovk$ of dimension $d$.
  We claim that we have a canonical  isomorphism
 $ H^{2d}_{\synt}(X,d)\simeq \Qp$. Indeed,  by Lemma \ref{long}, we have  $H^{2d}_{\synt}(X,d)=H^0_+(\ovk,H^{2d}_{\G}(X,d))$ and 
 \begin{align*}
 H^0_{+}(\ovk,H_{\G}^{2d}(X,d)) & =\kker((H^{2d}_{\hk}(X)\otimes_{F^{\nr}}\B^+_{\st})^{N=0,\phi=p^d}\to (H^{2d}_{\dr}(X)\otimes_{\ovk}\B^+_{\dr})/F^d)\\
 & =(H^{2d}_{\hk}(X)\otimes_{F^{\nr}}\B^+_{\st})^{N=0,\phi=p^d}
  \end{align*}
  The last equality holds because $F^dH^{2d}_{\dr}(X)=H^{2d}_{\dr}(X)$. Since $H^{2d}_{\hk}(X)\simeq F^{\nr}$ with Frobenius $p^d\phi$, we have 
  $(H^{2d}_{\hk}(X)\otimes_{F^{\nr}}\B^+_{\st})^{N=0,\phi=p^d}=\B_{\crr}^{+,\phi=1}=\Q_p$, as wanted.
  
   \subsubsection{Product of elliptic curves}
   We are now ready to do computations for products of certain elliptic curves.
   \begin{example} Let $E$ be an elliptic curve over $K$ with a supersingular reduction. Let $X=E\times E$.    
   \begin{enumerate}
   \item We have the following exact sequence of Galois representations
    $$
   0\to H^2_{\synt}(X_{\ovk},1)\to H^2_{\eet}(X_{\ovk},\Q_p(1))    \to C(-1)\to H^3_{\synt}(X_{\ovk},1)\to 0  
      $$
   \item If $\Sym^2 H^1_{\eet}(E_{\ovk})$ is an irreducible Galois representation, then   $$H^2_{\eet}(X_{\ovk},\Q_p(1))/ H^2_{\synt}(X_{\ovk},1)\simeq \Sym^2 H^1_{\eet}(E_{\ovk},\Q_p)(1),\quad H^3_{\synt}(X_{\ovk},1)\simeq C(-1)/\Sym^2 H^1_{\eet}(E_{\ovk},\Q_p)(1).$$ 
   \end{enumerate}
   \end{example} 
   \begin{proof}
   We note that we have the following exact sequence of Galois representations
   \begin{equation}
   \label{seqq}
   0\to H^2_{\synt}(X_{\ovk},1)\to  (H^2_{\hk}(X)\otimes_{F}\B_{\crr}^+)^{\phi=p}\lomapr{\iota_{\dr}} (H^2_{\dr}(X)\otimes_K\B_{\dr}^+)/F^1\to H^3_{\synt}(X_{\ovk},1)\to 0
   \end{equation}
It is obtained from the fundamental long exact sequence. The exactness on the left follows from Lemma \ref{long}. For the exactness on the right, we note, that using the
      K\"unneth formula in Hyodo-Kato cohomology and \ref{trace}, we get that
      $$
      (H^3_{\hk}(X)\otimes_{F}\B_{\crr}^+)^{\phi=p}=(H^1_{\hk}(E)\otimes_{F}\B_{\crr}^+)^{\phi=1} \oplus  (H^1_{\hk}(E)\otimes_{F}\B_{\crr}^+)^{\phi=1}. 
            $$
     Using the fact that the slope of the Frobenius on $H^1_{\hk}(E)$ is  $1/2$ we obtain
     $$  (H^3_{\hk}(X)\otimes_{F}\B_{\crr}^+)^{\phi=p}=(\B_{\crr}^{+,\phi=p^{-1/2}})^4 =0.   
           $$
    
      Consider now the following commutative diagram induces by the period maps
   $$\xymatrix{
   H^2_{\synt}(X_{\ovk},1)\ar@{^(->}[r]\ar@{^(->}[d]^{\rho_{\synt}} &(H^2_{\hk}(X)\otimes_{F}\B^+_{\crr})^{\phi=p} \ar[r]  \ar@{^(->}[d]^{\rho_{\hk}}&  (H^2_{\dr}(X)\otimes_K\B^+_{\dr})/F^1\ar@{->>}[r]\ar@{^(->}[d]^{\rho_{\dr}} &  H^3_{\synt}(X_{\ovk},1) \\
    H^2_{\eet}(X_{\ovk},\Q_p(1))\ar@{^(->}[r]\ar@{->>}[d] &  (H^2_{\hk}(X)\otimes_{F}\B_{\crr})^{\phi=p}\ar@{->>}[r]\ar@{->>}[d] &  (H^2_{\dr}(X)\otimes_K\B_{\dr})/F^1\ar@{->>}[d] \\\
   C(-1)\ar@{^(->}[r] & \coker{\rho_{\hk}}\ar@{->>}[r]^f & \coker{\rho_{\dr}}
   }
   $$
   The rows and columns are exact. The fact that $\ker(f)\simeq F^2H^2_{\dr}(X)\otimes_KC(-1)\simeq C(-1)$ follows from the fundamental exact sequence
   $$
   0\to \Q_pt\to \B^{+,\phi=p}_{\crr}\to \B^+_{\dr}/F^1\to 0
   $$
   Snake Lemma yields the following exact sequence
   $$
   0\to H^2_{\synt}(X_{\ovk},1)\to H^2_{\eet}(X_{\ovk},\Q_p(1))    \to C(-1)\to H^3_{\synt}(X_{\ovk},1)\to 0,  
      $$
     as claimed in the first part of our example. 
     
          We argue now for the second part. Since de Rham cohomology satisfies K\"unneth formula, the computations in \ref{trace} yield that $(H^2_{\dr}(X)\otimes_K\B_{\dr}^+)/F^1\simeq C$. In a similar way, using the
      K\"unneth formula in Hyodo-Kato cohomology and \ref{trace}, we get that
      $$
      (H^2_{\hk}(X)\otimes_{F}\B_{\crr}^+)^{\phi=p}=(H^1_{\hk}(E)^{\otimes 2}\otimes_{F}\B_{\crr}^+)^{\phi=p}\oplus \Q_p^2      
      $$
      and the vector space $\Q_p^2$ is in the kernel of $\iota_{\dr}$ (since $F^1H^2_{\dr}(E)=H^2_{\dr}(E)$). Hence we have the exact sequence
     $$
       0\to H^2_{\synt}(X_{\ovk},1)\to (H^1_{\hk}(E)^{\otimes 2}\otimes_{F}\B_{\crr}^+)^{\phi=p}\oplus \Q_p^2\to C\to H^3_{\synt}(X_{\ovk},1)\to 0     
     $$

     Using the fact that the slope of the Frobenius on $H^1_{\hk}(E)$ is  $1/2$ and $\B^{+,\phi=1}_{\crr}=\Q_p$, we compute that $
     (H^1_{\hk}(X)^{\otimes 2}\otimes_{F}\B_{\crr}^+)^{\phi=p}\simeq \Q_p^4$.      
   Since $H^1_{\hk}(E)^{\otimes 2}=\Sym^2 H^1_{\hk}(E)\oplus F$, we have  $$(H^1_{\hk}(E)^{\otimes 2}\otimes_{F}\B_{\crr}^+)^{\phi=p}=
   (\Sym^2 H^1_{\hk}(E)\otimes_F\B_{\crr}^+)^{\phi=p}\oplus \Q_p$$
    and, by \ref{thegroup},  $\Q_p$ is in the kernel of $\iota_{\dr}$.
   Hence we have the exact sequence
     $$
       0\to H^2_{\synt}(X_{\ovk},1)\to (\Sym^2 H^1_{\hk}(E)\otimes_{F}\B_{\crr}^+)^{\phi=p}\oplus \Q_p^3\to C\to H^3_{\synt}(X_{\ovk},1)\to 0     
     $$
   It follows that,   if $\Sym^2 H^1_{\eet}(E_{\ovk},\Q_p)$ is an irreducible Galois representation, then $H^2_{\synt}(X_{\ovk},1)\simeq \Q_p^3$ and $ H^2_{\eet}(X_{\ovk},\Q_p(1))/ H^2_{\synt}(X_{\ovk},1)\simeq \Sym^2 H^1_{\eet}(E_{\ovk},\Q_p)(1)$, as wanted.
   \end{proof}   
 \begin{remark}
   The above proof shows that $H^3_{\synt}(X_{\ovk},1) $ as a Galois representation is a quotient of $C(-1)$ and of $C$ by a finite rank $\Q_p$-Galois representation.   This type of curious phenomena was studied by Fontaine in \cite{FD}.
   \end{remark}
   \subsection{Galois descent}
Let $X$ be a variety over $K$. Then the syntomic complex $\R\Gamma_{\synt}(X_{\ovk},r)$ is equipped with a $G_K$-action. 
   Set\footnote{See \cite[4.2]{NN} for the necessary formalities concerning continuous Galois cohomology and Hochschild-Serre spectral sequence.} $\R\Gamma^{\natural}_{\synt}(X,r):=\R\Gamma(G_K,\R\Gamma_{\synt}(X_{\ovk},r))$.
Since, $$\R\Gamma_{\eet}(X,\Q_p(r))=\R\Gamma(G_K,\R\Gamma_{\eet}(X_{\ovk},\Q_p(r))),$$ we have the induced syntomic period map 
$$
\rho_{\synt}^{\natural}: \R\Gamma^{\natural}_{\synt}(X,r)\to \R\Gamma_{\eet}(X,\Q_p(r)).
$$
\begin{proposition}
There is a canonical distinguished triangle
 
 \begin{equation}
 \label{disting}
 \R\Gamma^{\natural}_{\synt}(X,r)\lomapr{\rho^{\natural}_{\synt}}  \R\Gamma_{\eet}(X,\Q_p(r))\to 
 \R\Gamma(G_K, (\R\Gamma^B _{\hk}(X_{\ovk})\otimes_{F^{\nr}}\B_{\crr}/\B^+_{\crr})^{\phi=p^r})
 \end{equation}

\end{proposition}
\begin{proof}
Consider  the following complex
$$
C^+(X_{\ovk}):=[\xymatrix@C=45pt{(\R\Gamma^B_{\hk}(X_{\ovk})\otimes_{F^{\nr}}\B^+_{\st})^{N=0,\phi=p^r} \ar[r]^-{\iota_{\dr}\otimes \iota} &   (\R\Gamma ^B_{\dr}(X_{\ovk})\otimes_{\ovk}\B^+_{\dr})/F^r]}
$$
representing $\R\Gamma_{\synt}(X_{\ovk},r)$. 
Define $C(X_{\ovk})$ by omitting the superscript $+$ in the above definition. 
We have \begin{align*}
C(X_{\ovk})/C^+(X_{\ovk})  =[A(\R\Gamma^B_{\hk}(X_{\ovk}))\veryverylomapr{\iota_{\dr}\otimes \iota}  B(\R\Gamma^B_{\dr}(X_{\ovk}))],
\end{align*}
where 
\begin{align*}
A(\R\Gamma^B_{\hk}(X_{\ovk})) & :=(\R\Gamma^B _{\hk}(X_{\ovk})\otimes_{F^{\nr}}\B_{\st}/\B^+_{\st})^{N=0,\phi=p^r},\\
B(\R\Gamma^B_{\dr}(X_{\ovk})) &:=  (\R\Gamma^B _{\dr}(X_{\ovk})\otimes_{\ovk}\B_{\dr})/F^r/ (\R\Gamma ^B_{\dr}(X_{\ovk})\otimes_{\ovk}\B^+_{\dr})/F^r.
\end{align*}

Hence
\begin{align*}
\R\Gamma(G_K,  C(X_{\ovk})/C^+(X_{\ovk}))
 =  [\R\Gamma(G_K,  A(\R\Gamma^B_{\hk}(X_{\ovk}))) 
   \veryverylomapr{\iota_{\dr}\otimes\iota}  \R\Gamma(G_K,B(\R\Gamma^B_{\dr}(X_{\ovk})))]
\end{align*}
and it suffices to show that the complex
$
\R\Gamma(G_K,  B(\R\Gamma^B_{\dr}(X_{\ovk})))$
is acyclic.  Since we have a quasi-isomorphism
$\R\Gamma^B_{\dr}(X)\otimes_K\ovk\stackrel{\sim}{\to}\R\Gamma^B_{\dr}(X_{\ovk})$, it suffices to show that, for a  finite rank  $K$-vector space $D_K$ with a descending  filtration such that $F^0D_K=D_K$,  the complex $\R\Gamma(G_K,B(D_K))$ is acyclic. But  the complex $B(D_K)$ has a natural filtration with graded pieces equal to copies of $C(j)$'s for strictly negative  $j$'s. Since  $H^*(G_K,C(j))=0,j < 0$, the acyclicity of the complex $\R\Gamma(G_K,B(D_K))$ follows. 

  \end{proof}

 \begin{remark}
It is not clear to us what is the relation between the distinguished triangle (\ref{disting})  and the Bloch-Kato (dual) exponential exact sequence.
\end{remark}
 
\end{document}